\documentclass{amsart}

\usepackage[colorlinks=True, citecolor=blue, linkcolor=blue]{hyperref}
\usepackage{amssymb}
\usepackage[normalem]{ulem}
\usepackage[utf8]{inputenc}
\usepackage{mathtools}
\usepackage[nameinlink]{cleveref}
\usepackage[dvipsnames]{xcolor}

\usepackage{braket}

\theoremstyle{definition}
\newtheorem{definition}{Definition}[section]

\theoremstyle{plain}
\newtheorem{theorem}[definition]{Theorem}
\newtheorem*{theorem*}{Theorem}
\newtheorem{lemma}[definition]{Lemma}
\newtheorem{proposition}[definition]{Proposition}
\newtheorem{corollary}[definition]{Corollary}

\DeclareMathOperator{\Tr}{Tr}

\DeclareMathOperator{\rk}{rk}

\DeclareMathOperator{\id}{id}

\DeclareMathOperator{\GL}{GL}

\DeclareMathOperator{\im}{im}

\DeclareMathOperator{\Nr}{Nr}

\newcommand{\QQ}{\mathbb{Q}}
\newcommand{\FF}{\mathbb{F}}

\newcommand{\ZZ}{\mathbb{Z}}

\renewcommand{\O}{\mathcal{O}}
\renewcommand\o{
  \mathchoice
    {{\scriptstyle\mathcal{O}}}%
    {{\scriptstyle\mathcal{O}}}%
    {{\scriptscriptstyle\mathcal{O}}}%
    {\scalebox{.7}{$\scriptscriptstyle\mathcal{O}$}}%
  }

\makeatletter
\newcommand{\bigperp}{%
  \mathop{\mathpalette\bigp@rp\relax}%
  \displaylimits
}

\newcommand{\bigp@rp}[2]{%
  \vcenter{
    \m@th\hbox{\scalebox{\ifx#1\displaystyle2.1\else1.5\fi}{$#1\perp$}}
  }%
}
\makeatother

\DeclareMathOperator{\scale}{\mathfrak{s}}
\DeclareMathOperator{\norm}{\mathfrak{n}}

\DeclareMathOperator{\Norm}{Nr}

\newcommand{\pO}{\mathfrak{P}}

\newcommand{\mc}[1]{\mathcal{#1}}

\newcommand{\Fi}{\mathbb F}
\newcommand{\floor}[1]{\left\lfloor #1\right\rfloor}

\newcommand{\sprod}[2]{\langle #1, #2 \rangle}
\newcommand{\sprodq}[2]{\langle #1 \rangle}

\newcommand{\fp}{\mathfrak{p}}
\newcommand{\FP}{\mathfrak{P}}
\newcommand{\FA}{\mathfrak{A}}
\newcommand{\FD}{\mathfrak{D}}

\newcommand{\RN}[1]{\textup{\uppercase\expandafter{\romannumeral#1}}}

\makeatletter
\newcommand*{\defeq}{\mathrel{\rlap{%
                     \raisebox{0.3ex}{$\m@th\cdot$}}%
                     \raisebox{-0.3ex}{$\m@th\cdot$}}%
                     =}

\title[]{Generation of Local Unitary Groups}

\author{Simon Brandhorst}

\address{Simon Brandhorst,
Fakult\"at f\"ur Mathematik und Informatik, Universit\"at des Saarlandes, Campus E2.4, 66123 Saarbr\"ucken, Germany}
\email{brandhorst@math.uni-sb.de}

\author{Tommy Hofmann}
\address{Tommy Hofmann,
Naturwissenschaftlich-Technische Fakult\"at, Universit\"at Siegen, Walter-Flex-Straße 3, 57068 Siegen, Germany}
\email{thofmann@mathematik.uni-kl.de}
\makeatother
\thanks{Gefördert durch die Deutsche Forschungsgemeinschaft (DFG) – Projektnummer 286237555 – TRR 195.
Funded by the Deutsche Forschungsgemeinschaft (DFG, German Research Foundation) – Project-ID 286237555 – TRR 195.}

\address{Sven Manthe,
Fakult\"at f\"ur Mathematik und Informatik, Universit\"at des Saarlandes, Campus E2.4, 66123 Saarbr\"ucken, Germany}
\email{sven.manthe@gmerek-manthe.de}

\makeatletter
\@namedef{subjclassname@2020}{%
  \textup{2020} Mathematics Subject Classification}
\makeatother

\date{\today}

\makeatletter
\let\@wraptoccontribs\wraptoccontribs
\makeatother

\begin{document}
\begin{abstract}
  Let $E$ be a two-dimensional \'etale algebra over a non-Archimedean local
  field $K$ of characteristic zero. We show that the unitary group of a
  non-degenerate hermitian lattice over $E$ is generated by symmetries and
  rescaled Eichler isometries. In the appendix we show that unless $E/K$ is a
  ramified dyadic field extension and the residue field has two elements,
  symmetries suffice.
\end{abstract}
\contrib[With an appendix by]{Sven Manthe}
\maketitle

\section{Introduction}
  The factorization of elements of a group over a distinguished set of generators
  is a classical topic in geometry and algebra. A prominent example of such a
  factorization is provided by Gaussian reduction, which shows that every invertible matrix over a field factors into elementary matrices.
  Similar results hold in other classical groups, for example, orthogonal or unitary groups
  of quadratic or hermitian spaces over fields (\cite{dieudonne1948}).
  For instance the Cartan--Dieudonné theorem asserts that an orthogonal
  transformation of an $n$-dimension symmetric bilinear space (in characteristic not two) factors into at most $n$ reflections.

  When extending the theory of classical groups from fields to arithmetic rings like rings of integers in local or global fields, factorization (or generation) theorems play once more an important role, for example, when enumerating isometry classes in a genus of quadratic or hermitian lattices.

  Much work has been spent on generalizing the theorem of Cartan--Dieudonn\'e to
  orthogonal groups of quadratic lattices over local fields of characteristic zero.
  If the local field is non-dyadic, that is, $2$ is a unit in the ring of integers,
  it was already observed and exploited by Kneser~\cite{Kneser1956} in the study of spinor genera that the orthogonal group is again generated by symmetries.
  As is typical for the theory of quadratic forms, the situation becomes much more involved over a dyadic field $K$. In~\cite{Omeara-Pollack1965} O'Meara and Pollack have shown that
  in case that $2$ is a prime element of $K$ and $K \neq \QQ_2$, the orthogonal group is still generated by symmetries, while for $K = \QQ_2$ this statement is false and additional generators (Eichler isometries) are necessary.
  On the other hand, investigating the case where $2$ is not a prime element anymore ``opens a Pandora's box of technicalities'' \cite{Omeara-Pollack1965b}.
  In this situation, only for unimodular lattices were O'Meara and Pollack able to show that
  the orthogonal group is generated by symmetries (and Eichler isometries for explicitly determined  exceptional cases).
  In the context of determining the integral spinor norms of orthogonal lattices over dyadic local fields, additional generation statements were obtained by Xu~\cite{Xu1995}. Though the general generation problem remains open.

  Recent developments in discrete holomorphic dynamics have
  renewed interest in hermitian lattices. Namely, automorphisms of complex surfaces with a given topological entropy or Siegel disks can be constructed using isometries of an integer lattice with a given characteristic polynomial \cite{mcmullen2016, OguisoYu2020}. These isometries are described in terms of hermitian lattices over certain number fields \cite{bayer-taelman2020, kirschmer2019unimod}. When considering local to global questions, a description of the unitary group at all local places is needed.

  In this paper, we solve the generation problem for unitary groups of hermitian lattices over non-Archimedean local fields of characteristic zero.
  More precisely we consider hermitian lattices over two-dimensional étale algebras.
  These are precisely the hermitian lattices obtained by completing a hermitian lattice over a number field at a prime ideal. The following theorem is our main result.
  \begin{theorem*}
    Let $E$ be a two-dimensional \'etale algebra over a non-Archimedean local field $K$ of characteristic zero and $L$ be a hermitian lattice over $E$.
    Then its unitary group $U(L)$ is generated by symmetries and rescaled Eichler isometries.
    If $E/K$ is not ramified dyadic, then symmetries suffice.
  \end{theorem*}
  By \Cref{symmetriesgenerators} in the appendix, symmetries suffice in the ramified dyadic case as well, unless the residue field has two elements.

  If $E/K$ is a non-dyadic field extension, it was shown by Böge~\cite{Boge1966} that the unitary group is generated by symmetries.
  For a ramified dyadic field extension $E/K$ with $2$ a prime element of $K$, Hayakawa~\cite{Hayakawa1968} proved that for $K \neq \QQ_2$ the same statement holds, while for $K = \QQ_2$ again (rescaled) Eichler isometries are necessary.
  In view of these partial results, the remaining cases, that we solve, are:
  \begin{itemize}
   \item the \emph{split} case, that is, $E = K \times K$,
   \item the \emph{dyadic inert} case, that is, $E/K$ is an unramified field extension of a dyadic field,
   \item the \emph{general ramified dyadic} case.
  \end{itemize}

  While the split and unramified cases can be dealt with uniformly with only minor difficulties using the
  techniques from~\cite{Boge1966},
  the general ramified dyadic case requires a different approach.
  A key ingredient is the classification of hermitian lattices due to Jacobowitz~\cite{Jacobowitz1962}, which allows us to reduce the general problem to lattices with controlled Jordan splittings, which can be dealt with case by case. This closes Pandora's box for the unitary group.

  The theorem yields an alternative proof for the computation of the determinant groups of unitary lattices due to Kirschmer \cite{Kirschmer2019}.
  Indeed, let $S(L)$ denote the subgroup of $U(L)$ generated by symmetries of a hermitian lattice $L$.
  Since rescaled Eichler isometries have determinant $1$, as a corollary of our result we obtain $\det(U(L)) = \det(S(L))$ which can be computed directly.
  Determinant groups play an important role in the enumeration of representatives of the isometry classes in genera of hermitian lattices over number fields.
  More precisely, the knowledge of the determinant groups allows the computation of the special genera in the genus of a hermitian lattice. Then strong approximation for the unitary group, as proved by Shimura \cite{shimura1964}, yields that if $L$ is indefinite, its special genus consists of a single isometry class only.

  The paper is organized as follows. After recalling the required background and notation in Section~\ref{sec:prelim}, the basic properties of symmetries and rescaled Eichler isometries are introduced in Section~\ref{sec:generators}.
  The split and unramified cases are dealt with in Section~\ref{sec:split}.
  In Section~\ref{sec:dyadic} the ramified dyadic case is addressed.
  The appendix answers the question of generation
  by symmetries in the ramified dyadic case.

 \section{Hermitian Lattices over local fields}\label{sec:prelim}
Let $K$ be a non-Archimedean local field of characteristic zero with valuation ring $\o$.
Recall that $K$ is \emph{dyadic} if $2 \not\in \o^\times$.
By $E$ we denote an étale $K$-algebra of dimension $2$.
Note that either $E$ is a quadratic field extension of $K$ or $E \cong K \times K$.
The non-trivial $K$-automorphism of $E$ with fixed field $K$ is denoted by $\overline{\phantom{x}} \colon E \to E$.
We denote by
\[ \Tr \colon E \longrightarrow K,\, \alpha \mapsto \alpha + \overline \alpha \quad \text{and} \quad \Nr \colon E \longrightarrow K, \,\alpha \mapsto \alpha \overline \alpha\]
the \emph{trace} and \emph{norm} of the $K$-algebra $E$ respectively
and by $\O$ the integral closure $\{ x \in E \mid \Nr(x) \in \o, \, \Tr(x) \in \o \}$ of $\o$ in $E$.
We write $\fp$ for the maximal ideal of $\o$ and $\FP$ for the largest proper ideal of $\O$ invariant under $\overline{\phantom{x}}$.
Note that if $E$ is a field, then $\O$ is the ring of integers of $E$ and $\FP$ its maximal ideal.
In case $E = K \times K$, we have $\O = \o \times \o$ and $\FP = \fp \O = \fp \times \fp $.
We denote by $\nu_\fp \colon K^\times \to \ZZ$ the normalized discrete valuation of $K$ with $\im(\nu_\fp) = \ZZ$, that is, $\nu_\fp(p)=1$ for a prime element $p$ of $K$.
Similar for $\nu_\FP$ in case $E$ is a field.

Let $V$ be a finitely generated free $E$-module. A \emph{hermitian form} on $V$ is
a map
\[\sprod{\cdot\,}{\cdot} \colon V \times V \longrightarrow E, \]
which is $E$-linear in its first argument and satisfies $\sprod{x}{y}=\overline{\sprod{y}{x}}$ for all $x,y \in V$.
We call the pair $(V,\sprod{\cdot\,}{\cdot})$ an \emph{hermitian space}. We define the shorthand $\sprodq{x}{x} \defeq \sprod{x}{x}$ for $x \in V$.
If the hermitian form is understood from the context, we will drop it from the notation and denote the hermitian space simply by $V$.
Two elements $x, y \in V$ are \emph{orthogonal} if $\sprod x y = 0$.
We call $u \in V$ isotropic if $\sprodq{u}{u} = 0$.
The \emph{unitary group} of a hermitian space $V$ is
\[U(V) = \{f \in \GL_E(V) \mid \sprod{f(x)}{f(y)}=\sprod{x}{y} \text{ for all $x, y \in V$}\}.\]

A \emph{hermitian $\O$-lattice} (or \emph{hermitian lattice}) $L$ is a finitely generated $\O$-submodule of a hermitian space $V$. We say that $L$ is a \emph{lattice} in $V$.
For $a \in K^\times$ we denote by $\prescript{a}{}{V}$ the rescaled hermitian space $(V, a \sprod{\cdot\,}{\cdot})$ and by $\prescript{a}{}{L}$ the hermitian $\O$-lattice $L$ in $\prescript{a}{}{V}$.
The \emph{unitary group} $U(L)$ of $L$ is defined as $U(V) \cap \GL(L)$.
Note that $U(\prescript{a}{}{L})=U(L)$.

If $S$ is a subset of $L$, then $S^\perp$ consists of the elements of $L$ orthogonal to $S$.
We call $L$ \emph{non-degenerate} if $L^\perp=0$ and \emph{degenerate} otherwise.
In what follows all hermitian lattices and spaces are assumed to be
non-degenerate unless stated otherwise. We say that $x \in L$ is
\emph{primitive} in $L$ if $\O x = E x \cap L$.

We denote by $\scale(L)=\sprod{L}{L} \subseteq E$ the \textit{scale} of $L$, which is a fractional $\O$-ideal, and
by $\norm(L)=\{\sprodq{x}{x} \mid x \in L\}\o \subseteq K$ the \emph{norm} of
$L$, which is a fractional $\o$-ideal. An orthogonal splitting of $L$ into sublattices $M$ and $N$ is denoted by $L=M \perp
N$. In this case we identify $U(M)$ with the subgroup $U(M) \times \{\id_N\}$
of $U(L)$. We use $L \cong M$ to state that two lattices $L$ and $M$ are
isometric.
If $G$ is a hermitian matrix over $E$, then we write $L = G$ to indicate that $G$ is the Gram matrix with respect to some basis of $L$.
In this case we define the \emph{determinant} of $L$ via $\det(L) = \det(G) \Nr(\O^\times)$, an element of $K^\times/\Nr(\O^\times)$.

Let $L$ be a full $\O$-lattice in $V$ and $\FA$ be a fractional ideal of $\O$.
Set
\[L^\FA = \{ x \in L \mid \sprod{x}{L}\subseteq \FA\}. \]
The lattice $L$ is called \textit{$\FA$-modular} if $L = L^\FA$. An $\O$-modular lattice is called \emph{unimodular}.
A decomposition $L = \bigperp_{i=1}^{t} L_i$ of $L$ is called a \textit{Jordan splitting} 
if the sublattices $L_i$ are $\FP^{s_i}$-modular such that $s_1 < s_2 < \dotsb < s_t$. Note that every hermitian lattice admits a Jordan splitting by \cite{Jacobowitz1962}.

\section{Symmetries and rescaled Eichler isometries}\label{sec:generators}
In this section we introduce the generators, the symmetries and rescaled
Eichler isometries, and derive first properties.
Let $E, K, \O$ and $\o$ be as in \Cref{sec:prelim}.

\begin{definition}
  Let $V$ be a hermitian space, $s \in V$ and $\sigma \in E^\times$ with $\sprodq{s}{s}=\Tr(\sigma)$. We call the linear map
  \[S_{s,\sigma} \colon V \longrightarrow V, \ x \longmapsto x -\sprod x s \sigma^{-1}s\]
  a \emph{symmetry} of $V$. It preserves the hermitian form $\sprod{\cdot\,}{\cdot}$.
  If $s$ is isotropic, then we have $\det(S_{s, \sigma}) = 1$ and
  otherwise $\det(S_{s,\sigma})=-\overline\sigma/\sigma$.
  The inverse is given by $S_{s,\sigma}^{-1}=S_{s,\bar \sigma}$.
 \end{definition}
Note that the symmetries are precisely those elements of the unitary group which fix a hyperplane.

\begin{lemma}\label{lem:symmetry}
  Let $L$ be a lattice in the hermitian space $V$.
  Let $S_{s, \sigma}$ be a symmetry of $V$ with $s \in L$ . Then $S_{s, \sigma}
  \in U(L)$ if $\sprod L s \subseteq \sigma \O$.
\end{lemma}

\begin{proof}
 We have to show that $S_{s,\sigma}(L) = L$. The inclusion $S_{s,\sigma}(L) \subseteq L$ follows from
 the definition and our assumptions. The other inclusion follows if we can show that $S_{s,\sigma}^{-1}(L) \subseteq L$. Since $S_{s,\sigma}^{-1} = S_{s,\bar \sigma}$, it suffices to show $\sigma \O = \bar \sigma \O$. By assumption $\Tr(\sigma) = \sprodq{s}{s} \in \sprod{L}{s} \subseteq \sigma \O$ and thus $\bar \sigma/{\sigma} \in \O$. Conjugation yields $\sigma/\bar \sigma \in \O$.
\end{proof}

\begin{lemma}\label{sym:act}
  Let $L$ be a lattice in the hermitian space $V$ and $x, x' \in L$ with $\sprodq x x = \sprodq{x'} {x'}$ and $\sprod{x}{x - x'} \neq 0$.
  Let $s = x - x'$ and $\sigma = \sprod{x}{ x - x'}$.
  Then $S_{s, \sigma} \in U(V)$ satisfies $S_{s,\sigma}(x) = x'$ and we have $S_{s, \sigma} \in U(L)$ if
  \[\sprod{L}{s}\sprod{x}{s}^{-1}s \subseteq L.\]
\end{lemma}

Recall that for $i \in \ZZ$, the \emph{hyperbolic plane} $H(i)$ of scale $\FP^i$ is by definition the hermitian $\O$-lattice with Gram matrix
\[ \begin{pmatrix} 0 & \pi^i \\ \bar \pi^i & 0 \end{pmatrix}, \]
where $\pi$ is a generator of $\FP$.
\begin{definition}
 Let $L$ be a hermitian lattice and $u,v \in L$ linearly independent. Set $P = \O u \oplus \O v$. We call $(u,v)$ a \emph{hyperbolic pair splitting $L$} if $u$ and $v$ are isotropic and $P \perp P^\perp=L$. We say that the hyperbolic plane $P \cong H(i)$ splits $L$.
\end{definition}

\begin{definition}
Let $L$ be a hermitian lattice, $(u,v)$ a hyperbolic pair splitting $L$ and $P = \O u \oplus \O v$.
Let $y \in P^\perp$ with $\sprod{L}{y} \subseteq \sprod{u}{v}\O$ and $\mu \in \O$ with
$\Tr(\mu \sprod u v )=-\sprodq y y$.
Then the \emph{rescaled Eichler isometry} $E_{y}^{\mu}\in U(L)$ is defined by
\[E_{y}^{\mu} \colon L \longrightarrow L, \, x \longmapsto x + \frac{\sprod x u}{\sprod v u}y + \left(\frac{\mu \sprod x u}{\sprod v u}-\frac{\sprod x y}{\sprod u v}\right)u.\] 

Let $S(L)$ denote the subgroup of $U(L)$ generated by symmetries and
$S(L) \subseteq X(L) \subseteq U(L)$ the subgroup generated by symmetries and rescaled Eichler isometries.
\end{definition}

\section{The unramified case}\label{sec:split}

Let $E, K, \O$ and $\o$ be as in \Cref{sec:prelim}.
In this section we consider the case where $E/K$ is unramified, that is,
$E \cong K \times K$ (the \emph{split} case) or $E/K$ is an unramified quadratic
extension of fields (the \emph{inert} case).
We prove that under this additional assumption, the unitary group of a hermitian lattice
is generated by symmetries.

Our strategy is as follows. By the results of Jacobowitz \cite{Jacobowitz1962} in the inert and split case,
a hermitian $\O$-lattice $L$ of scale $\FP^i$ is either of the form
$L = H(i) \perp N'$ or it is of the form $\O a \perp N$ with $\scale(N) \subseteq \FP^{i+1}$. We will consider both cases separately and show that $U(L) = S(L) U(N')$, respectively $U(L) = S(L) U(N)$.
Then the proof proceeds by induction on the rank of $L$.

\subsection{The inert case}
We begin with the inert case, that is, we assume that $E/K$ is an unramified quadratic extension.
Let $L$ be a hermitian lattice.
If $K$ is not dyadic, it was shown by B\"oge \cite[Satz 26]{Boge1966} that $U(L) = S(L)$. We now adapt the proof to the dyadic case.

For the rest of this subsection, we let $p \in \o$ be a prime element of $\o$, that is, $p\o = \fp$.
As $E/K$ is unramified, $p$ is also a prime element of $\O$.
Moreover the trace map $\Tr \colon \O \to \o$ is surjective (\cite[Sec. 5, Thm. 2]{Frohlich1967}) and there exists
$\rho \in \O$ such that $\Tr(\rho) = 1$. In particular $\rho \in \O^\times$.
Further there is an element $\omega \in \O^\times$ with $\Tr(\omega) = 0$. To see this note that $\Tr \colon E \rightarrow K$ has non-trivial kernel because $\Tr$ is $K$-linear and $\dim_K E=2>1$. If $\omega' \neq 0$ is in the kernel of $\Tr$, then $\omega = p^i \omega' \in \O^\times$ is a unit for $i = -\nu_\fp(\omega')$ and $\Tr(\omega) = p^i \Tr(\omega') = 0$.

\begin{lemma}\label{lem:inert-1}
  Let $u,u' \in L$ with $\sprod u L = \sprod{u'} L = \scale(L)$.
 Then there exists $x \in L$ with
 \[ \sprod x u \O = \sprod x {u'} \O = \scale(L).\]
 If in addition $\sprodq u u = \sprodq{u'}{u'} = 0$ and $\sprod u {u'}\O \neq \scale(L)$, then we may require that $\sprodq x x = 0$.
\end{lemma}

\begin{proof}
   By assumption there exist $y, y' \in L$ such that $\sprod u  y\O = \sprod{u'}{y'}\O = \scale(L)$. Then we can take $x$ as one of $y$, $y'$ or $y+y'$.
  Now assume that $u$ and $u'$ are isotropic.
  Consider $x' = x - \sprodq{ x}{ x}\rho \sprod u { x}^{-1}u$. As $\Tr(\rho) = 1$ this implies $\sprodq{x' }{ x'} = 0$.
  By construction $\sprod{u }{x'}\O = \sprod{u}{x}\O = \scale(L)$.
  Since $\sprod u {u'} \in \FP \scale(L)$ we also have $\sprod{u'}{x'}\O = \scale(L)$.
\end{proof}

\begin{lemma}\label{lem:refl-line}
  Let $L = \O a \perp N$ with $\sprodq a  a = 1$ and $\scale(N) \subseteq \FP$.
  Let $a' \in L$ be an element with $\sprodq{a'}{a'} = 1$. Then there exists a symmetry of $L$ mapping $a$ to $a'$.
\end{lemma}

\begin{proof}
  If $\sprod a {a - a'}\O = \sprod a  L = \O$,
  then from \Cref{sym:act} we have $S_{s,\sigma} \in U(L)$ and $S_{s, \sigma}(a) = a'$.

  Now assume that $\sprod a {a - a'}\O \subseteq \FP$.
  By \Cref{lem:inert-1} there exists $s \in L$ such that $\sprod a s \O = \sprod{a'} s\O = \scale(L) = \O$.
  In particular $\sprod s L  = \scale(L)$, so $s = \alpha a + n$ with $\alpha \in \O^\times$, $n \in N$; hence $\sprodq s s\O = \O$.
  Consider $S = S_{s,\sigma}$ with $\sigma = \rho \sprodq s s$. Then $\Tr(\sigma) = \sprodq s s$ and
  as $\sigma \in \O^\times$ and $\sprod L s = \O$, from \Cref{lem:symmetry} we know that $S \in U(L)$.
  Now
  \[ \sprod a {a - S(a')} = \sprod a {a - a'} + \sprod a {\sprod{a'} s \sigma^{-1}s} = \sprod a {a - a'} + \sprod a s \sprod s {a'}\bar \sigma^{-1}. \]
  Since $\sprod a {a - a'} \in \FP$, this implies $\sprod a {a - S(a')}\O = \sprod a L = \O$.
  Thus from the first part it follows that there exists a symmetry mapping $a$ to $S(a')$.
\end{proof}

\begin{lemma}\label{lem:rotation-reflection-inert}
 Every rescaled Eichler isometry of $L$ with respect to a hyperbolic plane $P$ is a product of two symmetries of $L$ and an element of $U(L)$ that fixes $P$ pointwise.
\end{lemma}
\begin{proof}
 Use the skew element $\omega$ and \cite[Hilfssatz 7]{Boge1966}.
\end{proof}

\subsection{The split case}
We now consider the case where $E = K \times K$.
Note that in this case
the involution is given by $\overline{\phantom{x}} \colon E \rightarrow E,\, (a,b) \mapsto (b,a)$
and we have $\O = \o \times \o$ as well as
$\FP = \fp \times \fp$.

\begin{lemma}\label{lem:isotropic-y}
 Let $u,u' \in L$ with $\sprod u L = \sprod{u'} L = \scale(L)$.
 Then there exists $x \in L$ with
 \[\sprod x u \O = \sprod x {u'}\O = \scale(L).\]
 If in addition $\sprodq u u = \sprodq{u'}{u'} = 0$ and $\sprod u {u'}\O \neq
 \scale(L)$, then we may require that $\sprodq x x = 0$.
\end{lemma}

\begin{proof}
 After rescaling the hermitian form we may assume that $\scale(L) = \O$.
 By assumption there exist $y,y' \in L$ with $\sprod y u \O= \sprod{y'}{u'}\O=\O$.
 There are $\lambda$ and $\mu$ in $\{0,1\}^2 \subseteq K \times K$ such that $x =\lambda y + \mu y'$
 satisfies
 \[\sprod x u \O = \sprod x {u'} \O=\O.\]
Now assume that $u$ and $u'$ are isotropic. By the first part there exists $\tilde{x}$ with $\sprod{\tilde{x}}{u}\O = \sprod{\tilde{x}}{u'}\O=\O$.
After multiplying $\tilde{x}$ by a unit, we may assume that $\sprod{\tilde{x}}{ u} = 1$.
Set $x = \tilde{x} + \alpha u$ for some $\alpha \in \O$ with
$\Tr(\alpha) = \sprodq{\tilde{x} }{\tilde{x}}$. Then $\sprodq{x}{x} = \sprodq{\tilde{x}}{\tilde{x}} + \Tr(\alpha) = 0$ and $\sprod{x} u  = \sprod{\tilde{x}}{ u} = 1$.
We have to choose $\alpha$ in such a way that \[ \sprod{x}{u'}= \sprod{\tilde{x}}{ u'} + \alpha \sprod{u}{u'} \in \O^\times.\]
If $\sprod{u}{u'} \in \fp \times \fp$ or $\sprod{u}{u'} \in \o^\times \times \fp$, we take $\alpha =(0,\sprodq{\tilde{x}}{\tilde{x}})$.
Otherwise $\sprod u {u'} \in \fp \times \o^\times$ and we may take $\alpha = (\sprodq{\tilde{x}}{\tilde{x}},0)$.
\end{proof}

\begin{lemma}\label{lem:symmetry-special}
  Let $L = a\O \perp N$ with $\sprodq a a = 1$ and $\scale(N) \subseteq \FP$. Let $a' \in L$ be an element with $\sprodq{a'}{a'} = 1$.
  If $|\o/\fp|\neq 2$, then there exists a product of symmetries which maps $a$ to $a'$.
\end{lemma}

\begin{proof}
  Note that the assumptions imply $\sprod a L = \sprod{a'} L = \O = \scale(L)$.

  If $\sprod a {a-a'}\O= \O$, then \Cref{sym:act} provides a symmetry that maps $a$ to $a'$.
  Otherwise assume that $\sprod a {a -a'}\in  \o \times \fp$.
  \Cref{lem:isotropic-y} provides $s \in L$ with \[\sprod s a \O = \sprod{s}{a'}\O= \O.\] Then $\sprod s L = \O $.
  By the shape of $L= a\O \perp N$, this implies that $\sprodq s s \in \O^\times$.
    Since  $|\o/\fp| \neq 2$,
 there exists $\epsilon \in \o^\times$ with
 $1 - \epsilon \in \o^\times$.
  Let $\sigma = \sprodq{s}{s} (\epsilon,1-\epsilon)$ or $\sigma =\sprodq{s}{s} (1-\epsilon,\epsilon)$. Then $\Tr(\sigma) =\sprodq s s$, and the symmetry $S_{s,\sigma}$ preserves $L$. We have
   \begin{equation}\label{eqn:special}
     \sprod a {a-S_{s, \sigma}(a')}= \sprod a {a-a'} + \sprod{s}{a'} \sprod a s \bar \sigma^{-1}.
   \end{equation}
   For $|\o/\fp|>3$ we can choose $\epsilon$ in such a way that $\sprod a {a-S_{s, \sigma}(a')}$ is a unit.
   For $|\o/\fp|=3$ write $a' = \alpha a + n $ with $\alpha \in \O$ and $n \in N$.
   Then
   \[1 = \sprodq{a'}{a'} = \alpha \bar \alpha \sprodq{a}{a} + \sprodq{n}{n} \equiv \alpha \bar \alpha \mod \FP \]
   implies that $\alpha \equiv (1,1) \mod \FP$ or $\alpha \equiv (-1,-1) \mod \FP$.
   Hence, using $\sprod{a}{a-a'} \in \o \times \fp$ it follows that $\sprod{a}{a-a'} = 1 - \bar \alpha \equiv 0\mod \FP$. Thus any choice of $\epsilon \in \o^\times$ with $1-\epsilon \in \o^\times$ works.
   It then follows from \Cref{sym:act} that there exists a symmetry which maps $a$ to $S_{s,\sigma}(a')$.
   %
   The case $\sprod a {a -a'}\in  \fp \times \o$ is analogous.
\end{proof}

\begin{lemma}\label{lem:symmetry-special-2}
  Let $L = a\O \perp N$ with $\sprodq a a = 1$ and $\scale(N) \subseteq \FP$. Let $a' \in L$ be an element with $\sprodq{a'}{a'} = 1$.
  If $|\o/\fp|= 2$, then there exists a product of symmetries which maps $a$ to $a'$.
\end{lemma}
\begin{proof}
   Write $a' = \alpha a + n$ for some $\alpha \in \O^\times$ and $n \in N$. Since $|\o/\fp|=2$, this implies that $\alpha \equiv 1 \mod \FP$. Thus $\sprod{a}{a-a'} = 1 - \bar \alpha \equiv 0 \mod \FP$.
  We have $\scale(N) = \FP^k$ with $k\geq 1$ and
  \[\sprod a {a-a'}\O = \fp^i \times \fp^j \mbox{ with } i,j \geq 1.\]

  First case:  ($i,j \leq k$).
  Then $\sprod a {a-a'}\O \supseteq \scale(N)$ and
  $L = \O a \perp N$ imply that $\sprod L {a-a'} = \sprod a {a-a'} \O$.
  Now \Cref{sym:act} provides us with a symmetry that maps $a$ to $a'$.

  Second case: ($i,j \geq 2$).
  As before \Cref{lem:isotropic-y} provides $s \in L$ with $\sprod s a \O = \sprod{s}{a'}\O= \O$. Then $\sprod s L = \O = \scale(L)$
  and $\alpha:=\sprodq s s \in \o^\times$. Set $\sigma = \alpha p^{-1}(1+p,-1)$. Then $\Tr(\sigma) =\alpha$
  and $\sprod L s \O=\O\subseteq  p^{-1}\O=\sigma \O$. Thus $S_{s,\sigma}$ preserves $L$.
  We obtain
  \[\sprod a {a-S_{s, \sigma}(a')}= \sprod a {a-a'} + \sprod{s}{a'} \sprod a s \bar\sigma^{-1}.\]
  As $\sprod s{a'} \sprod a s\bar\sigma^{-1}\O= \FP$
  and $\sprod a {a - a'} \in \FP^2 $, this implies $\sprod a {a - S_{s,\sigma}a'} \O = \FP$.
  By the first case we find a symmetry mapping $S_{s,\sigma}(a')$ to $a$.

  Third case: ($i < k < j$).
    We choose $s \in L$, $\alpha = \sprodq s s \in \o^\times$ as in the third case and set $\sigma = \alpha p^{-k}(1 + p^k, -1)$.
    As before, $S_{s, \sigma}$ preserves $L$. Since $\sprod a {a - a'} = \fp^i \times \fp^j$
    and $\sprod s{a'} \sprod a s \bar \sigma^{-1}\O = \FP^k$, the
    assumption $i < k < j$ implies $\sprod a {a - S_{s, \sigma}a'} \O = \fp^i \times \fp^k$.
    By the first case we find a symmetry mapping $S_{s, \sigma}(a')$ to $a$.

    Fourth case: ($k=1$, $i=1$ and $j>1$).
  Since $a' = \alpha a + n$, we have $\sprodq n n = 1-\alpha \bar \alpha$.
  Set $u = (1,\alpha \bar \alpha-1)a+n$.
  Then $u$ is isotropic and we %
  have
  \[ \sprod a u  = (\alpha \bar \alpha -1,1) \quad \mbox{
  and }\quad \sprod{u}{a'} = (1,\alpha \bar \alpha-1) \bar \alpha + 1-\alpha \bar \alpha.\]
    Note that by assumption $(1 - \bar \alpha)\O = \sprod a {a - a'}\O = \fp \times \fp^j$ and therefore %
    \[\bar \alpha \equiv (1 + p,1) \mod \FP^2 \quad \text{and}\quad \alpha \bar \alpha -1 \equiv p \mod \FP^2;\]
    hence
    \begin{eqnarray*}
     \sprod a u \sprod u {a'} %
     &\equiv & (p,0) \mod \FP^2.
    \end{eqnarray*}
  With $\sigma = (1,-1)$, we obtain the symmetry $S_{u,\sigma}$ and calculate that
  \begin{eqnarray*}
    \sprod a {a-S_{u,\sigma}(a')}%
                &=& \sprod a {a-a'} + \sprod u {a'} \sprod a u \bar \sigma^{-1}.
  \end{eqnarray*}
  Since $\sprod a {a-a'} \equiv (p,0) \mod \fp^2\O$, we have $\sprod a {a-S_{u, \sigma}(a')} - \sprod u {a'} \sprod a u \bar\sigma^{-1} \equiv 0 \mod \fp^2\O$.
  By the second case, we find a product of symmetries which maps $a$ to $S_{u,\sigma}(a)$.
  The case ($k=1$, $i>1$ and $j=1$) is dealt with analogously.
\end{proof}

\begin{lemma}\label{lem:rotation-reflection}
 Every rescaled Eichler isometry of $L$ with respect to a hyperbolic plane $P$ splitting $L$ is a product of two symmetries and an element of $U(L)$ that fixes $P$ pointwise.
\end{lemma}
\begin{proof}

  Let $E_{y}^{\mu}$ be a rescaled Eichler isometry with respect to the hyperbolic pair $(u, v)$ splitting $L$. We may assume $\sprod u v =1$ so that in particular $\scale(L) = \O$.
  We have $E_{y}^{\mu}(v) = \mu u + v + y$ and \[\sprod{v -E_{y}^{\mu}(v)} v \O = \mu \O = \mu \scale(L).\]
  If $\mu$ is a unit, then \Cref{sym:act} yields a symmetry $S_{s,\sigma}$
  that maps
  $v$ to $E_{y}^{\mu}(v)$ and fixes $u$.

If $\mu$ is not a unit and $|\o/\fp|\neq 2$, then we claim that we can find
$\omega \in \O^\times$ with
\begin{enumerate}
  \item[(a)] $\Tr(\omega) =0$ and
  \item[(b)] $\mu - \omega^{-1} \in \O^\times$.
\end{enumerate}
Condition (a) forces $\omega = (t,-t)$ for some $t \in \o^\times$. If $\mu \in \fp \times \fp$, then $t=1$ suffices.
But if $\mu \in \o^\times \times \fp$ (or $\fp \times \o^\times)$,
then
we can take $t = \epsilon \in \o^\times$ such that $\mu \pm \epsilon^{-1}$ is a unit, since $|\o/\fp| \neq 2$.
Now by (b), using \Cref{sym:act} we find a symmetry mapping $v$ to
\[S_{u,\omega}(E_{y}^{\mu}(v)) = (\mu-\omega^{-1})u + v + y.\]
And since $\omega \in \O^\times$, the symmetry $S_{u,\omega}$ preserves $L$.

Suppose now that $\mu$ is not a unit and $|\o / \fp|=2$.
If $\mu \in \fp \times \fp$, we can argue as before.
Assume that  $\mu \in \o^\times \times \fp$.
Then $-\sprodq{y}{y}=\Tr(\mu) \in \o^\times$.
For $s = u + (0,1) y$ and $\sigma =(-1,1)$ the symmetry $S_{s,\sigma}$ preserves $L$ and fixes $u$. We have
\[\sprod{E_{y}^{\mu}(v)}{s} = \sprod{v + \mu u + y }{u + (0,1) y }= 1 + (1,0)\sprodq{y}{y} \]
as well as 
\[S_{s,\sigma}(E_{y}^{\mu}(v)) =  \mu' u + v + y'\]
with $\mu' = (\mu + (1,-1) + (1,0)\sprodq{y}{y})$,
and $y'=(1,0)y$.
Hence $\mu' \in \O^\times$ is a unit, and we are back to the first case. The case that $\mu \in \fp \times \o^\times$ is dealt with analogously.
\end{proof}

\subsection{Generation of unitary groups}

Let now be $E/K$ be inert or split and $L$ a hermitian $\O$-lattice.

\begin{proposition}\label{prop:isotropic-symmetries}
 Let $u,u' \in L$ be isotropic with $\sprod{u}{L}=\sprod{u'}{L}=\scale(L)$.
 Then there is a product of symmetries which maps $u$ to $u'$.
\end{proposition}
\begin{proof}
 If $\sprod{u}{u'}\O=\scale(L)$, then we can find a symmetry mapping $u$ to $u'$ by \Cref{sym:act}. Otherwise, by \Cref{lem:inert-1} in the inert case and by \Cref{lem:isotropic-y} in the split case, we find an isotropic $y \in L$ with $\sprod{u}{y}\O = \sprod{u'}{y}\O = \scale(L)$. Then there are symmetries mapping $u$ to $y$ and $y$ to $u'$.
\end{proof}

\begin{proposition}\label{prop:hyperbolic-refl}
  Let $P \subseteq L$ be a hyperbolic plane with $\scale(P) =\scale(L)$. Then $U(L) = S(L)U(P^\perp)$.
\end{proposition}
\begin{proof}
  Let $P = \O u \oplus \O v$ with $(u,v)$ a hyperbolic pair splitting $L$. We may assume that $\sprod u v =1$. Let now $\varphi \in U(L)$. By \Cref{prop:isotropic-symmetries} we find a product of symmetries $S$ with $S(\varphi(u)) = u$.
  We can write $v' = S(\varphi(v))$ as $v' = \mu u + \lambda v +y$ for some $y \in P^\perp$ and $\lambda, \mu \in \O$. Since
  $1 = \sprod{S(\varphi(u))}{S(\varphi(v))} = \sprod u {S(\varphi(v))}$, we have $\lambda =1$. Further $0 = \sprodq{S(\varphi(v))}{S(\varphi(v))} = \Tr(\mu) + \sprodq y y$. Thus $E_{y}^{\mu} \in U(L)$. Since $E_{y}^{\mu}(u) = u$ and $E_y^\mu(v) = v'$, the claim follows from \Cref{lem:rotation-reflection-inert,lem:rotation-reflection}.
\end{proof}

\begin{proposition}\label{prop:line-symmetries}
  Let $L = a\O \perp N$ with $\sprodq a a = 1$ and $\scale(N) \subseteq \FP$.
  If $a' \in L$ is an element with $\sprodq{a'}{a'} = 1$, then there exists a product of symmetries which maps $a$ to $a'$.
\end{proposition}

\begin{proof}
  In the inert case this follows from \Cref{lem:refl-line}, while in the split case this follows from
  \Cref{lem:symmetry-special,lem:symmetry-special-2}
\end{proof}

\begin{theorem}\label{thm:splitinert-symmetry}
  If $E/K$ is split or inert, then
the unitary group $U(L)$ is generated by symmetries.
\end{theorem}

\begin{proof}
  Since rescaling $L$ does not change the unitary group, we may assume that $\scale(L) = \O$.
  Now write $L = N \perp M$ with $N$ an $\scale(L)$-modular sublattice and $\scale(M) \subseteq \FP$.
  If $\rk N = 1$, then from \Cref{prop:line-symmetries} it follows that $S(L)U(N^\perp) = U(L)$.
  On the other hand if $\rk N \geq 2$, then $L$ contains an $\scale(L)$-modular hyperbolic plane $P$ and by \Cref{prop:hyperbolic-refl} we obtain $S(L)U(P^\perp)=U(L)$.
  In both cases we proceed by induction.
\end{proof}

\section{Ramified dyadic}\label{sec:dyadic}
We now consider the situation where $E/K$ is a quadratic ramified dyadic extension. We denote by $\pi$ a fixed prime element of $\O$ and set $p = \pi \bar \pi$, which is a prime element of $\o$.
Let $\FD$ be the different of the extension $E/K$, that is, $\FD^{-1} = \{ x \in E \mid \Tr(x \O) \subseteq \o\}$.
Denote by $e \in \ZZ$ the valuation of the different $\FD$, that is, $\FD = \mathfrak P^e$.
From~\cite[Ch. V, \S 3, Lemma 3]{Serre1979} it follows that for all $i \in \ZZ$, $i \geq 0$, we have
\[ \Tr(\mathfrak P^i) = \mathfrak p^{\lfloor (i + e)/2 \rfloor}.\]

If $e = 2 \nu_\fp(2) + 1$, then there exists a prime element $\eta \in \O$ with $\Tr(\eta) = 0$.
Moreover for $\rho = 1/2$ we have $\Tr(\rho) = 1$ and $\rho \O = \FP^{1-e}$.
Otherwise $e$ is even and there is a unit $\eta \in \O^\times$ with $\Tr(\eta) = 0$ and $\rho \in E^\times$
with $\Tr(\rho) = 1$ and $\rho \O = \FP^{1-e}$.
The existence of these elements and their valuations follow from the theory of local quadratic dyadic field extensions (see for example~\cite[Section 3]{Kirschmer2019} and take $\rho = (1+\alpha)/2$ using the notation of \cite{Kirschmer2019}).

In either case there exists $\rho \in E^\times$
with $\Tr(\rho) = 1$ and $\rho \O = \FP^{1-e}$ and a skew element $\omega \in E^\times$ with $\Tr(\omega)=0$ and $\nu_\FP(\omega)\equiv e \mod 2$. In what follows we make frequent use of the elements $\rho$ and $\omega$. Note that one of the complications of the dyadic case is that there is no unit of trace one. When setting up a symmetry we have to work with $\rho$ instead.

Let $L$ be a hermitian lattice. Then norm and scale satisfy the following inclusions
\[\scale(L)\FD \subseteq \norm(L)\O \subseteq \scale(L) .\]
We call $L$ \emph{normal} if $\norm(L)\O = \scale(L)$ and \emph{subnormal} otherwise.

We now recall the classification of hermitian lattices in the dyadic case.
Recall that two Jordan splittings $\bigperp_{i=1}^{t} M_i$ and $\bigperp_{j=1}^{t'} N_j$ have the same \emph{Jordan type} if $t = t'$ and for all $1 \leq i \leq t$ we have $\rk M_i = \rk N_i$, $\scale(M_i) = \scale(N_i)$ and $M_i$ is normal if and only if $N_i$ is normal.
For two elements $\alpha, \beta \in K^\times/\Nr(\O^\times)$ with $\alpha \o = \beta\o$ and a non-zero ideal $\mathfrak a$ of $\o$ we write $\alpha/\beta \cong 1 \mod \mathfrak a$ if there exists $\gamma \in \mathfrak a$ such that $\alpha = \beta (1 + \gamma)$ in $K^\times/\Nr(\O^\times)$.

\begin{theorem}[{\cite[Theorem 11.4]{Jacobowitz1962}}]\label{dyadic:classification}
Let $M = \bigperp_{i=1}^t M_i$ and $N = \bigperp_{i=1}^{t'} N_i$ be Jordan
splittings of hermitian $\O$-lattices. Then $M$ and $N$ are isometric if and only if the following four conditions hold:
\begin{enumerate}
  \item
    $M$ and $N$ have the same Jordan type,
  \item
    $\det(M)/{\det(N)} \in \Nr(\O^\times)$,
  \item
    $\mathfrak n_i := \norm(M^{\scale(M_i)}) = \norm(N^{\scale(N_i)})$ for $1 \leq i \leq t$,
  \item
    $\det(M_1 \perp \dotsb \perp M_i)/{\det(N_1 \perp \dotsb \perp N_i)} \cong 1 \mod \o \cap \mathfrak n_i \mathfrak n_{i+1}\scale(M_i)^{-2}$
    for $1 \leq i < t$.
\end{enumerate}
\end{theorem}

Recall that for an element $a \in \o$, the \emph{normic defect} of $a$ is defined as the fractional $\o$-ideal
\[ \mathfrak d_E(a) = \bigcap_{\beta \in E}(a - \Nr(\beta))\o. \]
Let now $(1+u_0) \in \o^\times$ be an element with normic defect $\mathfrak{d}_E(1+u_0)=\fp^{e-1}=u_0\o$ such that $\o^\times=\Nr(\O^\times)\cup (1+u_0) \Nr(\O^\times)$ (see~\cite[Section 6]{Johnson1968}). For $i \geq 0$ and $i + e \geq 2k \geq i$ define the following $\FP^i$-modular planes of norm $\fp^k$:
\[H(i)=\begin{pmatrix}
                0 & \pi^i \\ \bar\pi^i &0
               \end{pmatrix}, \qquad
H(i,k)=\begin{pmatrix} p^k & \pi^i \\ \bar \pi^i & 0 \end{pmatrix},\qquad
A(i,k)=\begin{pmatrix} p^k & \pi^i \\ \bar \pi^i & -u_0p^{i-k} \end{pmatrix}. \]
Moreover we have $H(i, k) \cong H(i)$ if and only if $k = \lfloor (i + e)/2 \rfloor$.
For a lattice isometric to $H(i)$, $H(i, k)$ or $A(i, k)$, a tuple $(x, y)$ of elements is
called a \emph{standard basis} if the Gram matrix with respect to $(x, y)$ has
the respective form.  We call any hermitian lattice isomorphic to $H(i)$ a
\emph{hyperbolic plane} of scale $\FP^i$.

\begin{corollary}[{\cite{Jacobowitz1962}, \cite[Corollary 3.3.20]{KirschmerHabil}}]\label{dyadic:modular}
Let $L$ be a $\FP^i$-modular lattice of norm $\fp^k$ and $V=L \otimes E$.
\begin{enumerate}
 \item Suppose $m=2r+1$ is odd. Then $i$ is even and
 \[L \cong \langle u p^{i/2}\rangle \perp H(i)^r \quad \mbox{ where } \quad u \Nr(\O^\times) = \det(L)p^{-mi/2}\Nr(\O^\times).\]
 \item Suppose $m=2r+2$ is even and $V$ is hyperbolic. Then
\[\pi^i \FD \subseteq \fp^k\O \subseteq \pi^i \O \mbox{ and } L \cong H(i,k)\perp H(i)^r . \]
\item Suppose $m=2r+2$ is even and $V$ is not hyperbolic. Then
\[\pi^i \FD \subsetneq \fp^k \O \subseteq \pi^i \O \mbox{ and }
L\cong A(i,k) \perp H(i)^r .\]
\end{enumerate}
\end{corollary}

We will also make use of the following two basic results about the norm map
modulo prime ideal powers.

\begin{lemma}\label{dyadic:norm-mod}
 Let $a \in \o^\times$. Then there exists $\epsilon \in \O^\times$ with $a \equiv \epsilon \bar \epsilon \mod \fp^{e-1}$.
\end{lemma}

\begin{proof}
  We have $\o^\times = \Nr(\O^\times) \cup (1+u_0)\Nr(\O^\times)$.
  If $a \in \Nr(\O^\times)$ is a norm, the statement is true.
  Otherwise there exists $\epsilon \in \O^\times$ with $a = (1 + u_0)\epsilon\bar\epsilon$ and $u_0 \in \fp^{e-1}$ yields
\[a =(1+u_0) \epsilon \bar \epsilon\equiv \epsilon \bar \epsilon \mod \fp^{e-1}. \qedhere \]
\end{proof}

\begin{lemma}\label{dyadic:goingdown}
  Let $\alpha \in \O$. Then $\alpha \equiv 1 \mod \FP$ if and only if $\alpha \bar \alpha \equiv 1 \mod \fp$.
\end{lemma}

\begin{proof}
Since $E/K$ is quadratic and ramified, the involution induces the identity on
$\O/\FP$.  Hence $\alpha^2 \equiv \alpha \bar \alpha \mod \FP$.
  If $\alpha \bar \alpha \equiv 1 \mod \fp$, then $\alpha^2 -1\equiv 0 \mod \FP$; hence $\alpha \equiv 1 \mod \FP$.
  If $\alpha \equiv 1 \mod \FP$, this gives $\alpha \bar \alpha \equiv 1 \pmod{ \FP \cap \o = \fp}$.
\end{proof}

\subsection{Hyperbolic planes}
In this section, we prove that symmetries and rescaled Eichler isometries act transitively on the collection of hyperbolic planes of a given scale splitting a hermitian lattice. If $K/\QQ_2$ is unramified, then this is a result of Hayakawa \cite{Hayakawa1968}. In the following we show that if $K/\QQ_2$ is ramified, the same ideas still work with only minor modifications.

\begin{lemma}\label{dyadic:hyperbolic-plane}
  The unitary group of a hyperbolic plane $L = H(i)$ is generated by symmetries.
\end{lemma}

\begin{proof}
  Let $(u, v)$ be a standard basis of $L=H(i)$. Consider $\varphi \in U(L)$ and write $\varphi(u) = \alpha u + \beta v$, $\varphi(v) = \gamma u + \delta v$  with $\alpha, \beta, \gamma, \delta \in \O$.
  If $\beta \in \O^\times$ is a unit, then
  \[ \sprod u {u - \varphi(u)}\O = \sprod{u}{\beta v} \O = \sprod u v \O = \scale(L). \]
  Hence by \Cref{sym:act} there exists a symmetry $S \in U(L)$ with $S(u) = \varphi(u)$.

  Now consider $\varphi' = S^{-1} \circ \varphi$, which satisfies $\varphi'(u) = u$. Write $\varphi'(v) = \gamma' u + \delta' v$.
  Then
  \[ \delta' \sprod v u = \sprod{\varphi'(v)} u = \sprod{\varphi'(v)}{\varphi'(u)} =  \sprod v u, \]
  and hence $\delta' = 1$. Now
  \[ 0 = \sprodq{\varphi'(v)}{\varphi'(v)} = \Tr(\gamma' \sprod u v) \]
  and therefore $\Tr(\overline \gamma' \sprod v u) = 0 = \Tr(\gamma'^{-1} \sprod v u)$.
  Now setting $\sigma = -\gamma'^{-1}\sprod v u$ we have $\Tr(\sigma) = 0 =\sprodq u u$, $\sigma^{-1}\sprod{L}{u} \O= \gamma'\O \subseteq \O$
  as well as $S_{u, \sigma} \in S(L)$ and
  \[ S_{u, \sigma}(v) = v + \sprod v u \sprod v u ^{-1} \gamma' u = \varphi'(v), \quad S_{u, \sigma}(u) = u. \]
  This shows that $\varphi \in S(L)$. By swapping $u$ and $v$ the same holds when $\gamma \in \O^\times$.

  Let $H$ be the set of $\varphi \in U(L)$ where $\beta, \gamma \in \FP$. Then $H$ is a subgroup of $U(L)$ and the above argument shows that $U(L) \setminus H \subseteq S(L)$.
  It is sufficient to show that $H \subseteq S(L)$.
  To this end consider the map $\varphi$ defined by $\varphi(u) = v$ and $\varphi(v) = \bar \pi^i \pi^{-i} u$, which satisfies $\varphi \in U(L) \setminus H$.
  Now let $h \in H$. Then $\varphi \circ h \not\in H$ and hence $\varphi \circ h \in S(L)$.
  It follows that $h \in S(L)$.
\end{proof}

\begin{lemma}\label{dyadic:rot1}
  Let $(v,u_1), (v,u_2)$ be hyperbolic pairs splitting the hermitian lattice $L$ with $\sprod{v} {u_1} = \sprod{v}{u_2}$.
  Then there exists a rescaled Eichler isometry $\varphi \in X(L)$ such that $\varphi(u_1) = u_2$ and $\varphi(v) = v$.
\end{lemma}

\begin{proof} 
  Write $L = (\O {u_1} \oplus \O v) \perp M$ for some sublattice $M$ of $L$. We can write
${u_2} = \alpha {u_1} + \beta v + w$ with $\alpha, \beta \in \O$, $w \in M$. Note that $\sprod{w}{L}=\sprod{w}{M}= \sprod{u_2}{M} \subseteq \sprod{u_2}{L} = \sprod{v}{u_2}\O=\sprod{v}{u_1}\O$.
  Then $\sprod{u_2} v = \alpha \sprod{u_1}{v}$ and hence $\alpha = 1$.
  Now
  \[ 0 = \sprodq{u_2} {u_2} = \Tr(\beta \sprod v {u_1}) + \sprodq w w.\]
  Hence $\Tr(\beta \sprod {v}{u_1}) = -\sprodq w w$ and the rescaled Eichler isometry $\varphi = E_{w}^{\beta}$ with respect to $(v,u_1)$
  satisfies $\varphi({u_1}) = {u_1} + \beta v + w = {u_2}$ and $\varphi(v) = v$.
\end{proof}

\begin{lemma}\label{dyadic:rot2}
  Let $({u_1}, v_1)$, $(u_2, v_2)$ be hyperbolic pairs splitting a hermitian lattice $L$ with $\sprod{u_1}{v_1} = \sprod{u_2}{v_2}$.
  Then there exists $\varphi \in X(L)$ and $\varepsilon \in \O^\times$ such that $\varphi({u_1}) = \varepsilon u_2$
  or $\varphi(v_1) = \varepsilon v_2$.
\end{lemma}

\begin{proof}
  Write $L = (\O {u_1} \oplus \O v_1) \perp M$ and $u_2 = \alpha {u_1} + \beta v_1 + w$, $v_2 = \gamma {u_1} + \delta v_1 + w'$ with $\alpha,\beta,\gamma,\delta \in \O$ and $w, w' \in M$.
  If $\beta \in \O^\times$ then
  \[\sprod{u_1} {u_1 - u_2}\O = \sprod{u_1} {u_1 - \alpha u_1 - \beta v_1 + w}\O = (-\bar\beta\sprod{u_1}{v_1})\O = \sprod{u_1}{v_1}\O=\sprod{L}{u_1-u_2}\]
  and by \Cref{sym:act} there exists $S \in S(L)$ such that $S({u_1}) = u_2$.
  Similarly for the case $\delta \in \O^\times$, we find $S \in S(L)$
  with $S(v_1)=v_2$.

  If $\alpha \in \O^\times$, then $\alpha^{-1} u_2 = {u_1} + \beta \alpha^{-1}v_1 + \alpha^{-1}w$,
  \[ \sprod{\alpha^{-1} u_2} {v_1} 
  = \sprod{u_1}{v_1}. \]
  and $(v_1, \alpha^{-1}u_2)$ is a hyperbolic pair splitting $L$.
  From \Cref{dyadic:rot1} it follows that there exists $T \in X(L)$ such that $T(u_1) = \alpha^{-1}u_2$.
  Similarly for $\gamma \in \O^\times$.

  We now assume that $\alpha,\beta,\gamma,\delta \in \FP$.
  In particular $w, w' \in M$ are primitive vectors with $\sprod{w}{L}=\sprod{w'}{L}=\sprod{u_1}{v_1}\O$.
  Now $0 = \sprodq{u_2}{u_2} = \Tr(\alpha \bar \beta \sprod{u_1}{v_1} ) + \sprodq w w$.
  Since $\alpha \bar \beta \sprod{u_1}{v_1} \in \FP^{2}\sprod{u_1}{v_1}$, we have
  \[\sprodq w w \in \fp\Tr(\sprod{u_1}{v_1}\O) \quad \mbox{and similarly}\quad \sprodq{w'} {w'} \in \fp\Tr(\sprod{u_1}{v_1}\O).\]
  On the other hand
  $\sprod w {w'} \equiv \sprod{u_2 }{v_2} \mod \FP^2 \sprod{u_1}{v_1}=\FP^2\sprod{u_2}{v_2}$,
  which implies that \[\sprod w {w'} \O = \sprod{u_2 }{v_2} \O =\sprod{u_1}{v_1} \O.\]
  Let $W = \O w \oplus \O w'$. Then
  \[\scale(W)=   \sprodq{w}{w}\O+\sprod{w}{w'}\O +\sprodq{w'}{w'}\O= \fp \Tr(\sprod{u_1}{v_1}\O) + \sprod{u_1}{v_1}\O=\sprod{u_1}{v_1}\O.\]
  Also
  \[\norm(W) = \sprodq{w}{w}\o + \Tr(\sprod{w}{w'}\O)  = \Tr(\sprod{u_1}{v_1}\O).\]
  As $2 + 2\lfloor (i + e)/2 \rfloor \geq i + e$, we have
  \[ \det(W) = \sprodq{w}{w}\sprodq{w'}{w'}- \sprod{w}{w'}\sprod{w'}{w} \equiv -\Nr(\sprod{w}{w'}) \mod \mathfrak p^{i + e}. \]
  Since $\Nr(\sprod w {w'})\o = \fp^i$ this implies $-\det(W)/\Nr(\sprod w {w'}) \equiv 1 \mod \fp^e$ and hence by \cite[XV, \S{}2, Cor. 2 to Th. 1]{Serre1979} we know that
  $-\det(W)/\Nr(\sprod w {w'}) \in \Nr(\O^\times)$.
  By \Cref{dyadic:modular} the lattice $W$ is a hyperbolic plane.
  Since $\sprod{W}{L}= \sprod{w}{w'}\O=\scale(W)$, the sublattice $W$ is an orthogonal summand of $M$.

  Hence we can find an isotropic vector $z \in W$ with $\sprod w z = \sprod{u_1}{v_1}$.
  Then
  \[\sprod{u_1} {v_1 + z} = \sprod{(\alpha + 1)^{-1} u_2}{v_1 + z}.\]
  As further $({v_1} + z,{u_1})$ and $(v_1+z, (\alpha + 1)^{-1}u_2$ are hyperbolic pairs splitting $L$, it follows from \Cref{dyadic:rot1} that there exists $T \in X(L)$ such that $T({u_1}) = (\alpha + 1)^{-1} u_2$.
\end{proof}

\begin{proposition}\label{dyadic:hyperbolic}
  Let $(u_1,v_1)$ and $(u_2,v_2)$ be hyperbolic pairs splitting a hermitian lattice $L$. If $\sprod{u_1}{v_1} = \sprod{u_2}{v_2}$,
  then there exists $T \in X(L)$ such that $T(u_1) = u_2$ and $T(v_1) = v_2$.
\end{proposition}

\begin{proof}
  By \Cref{dyadic:rot2} there exists $\varepsilon \in \O^\times$ and $\varphi \in X(L)$ such that $\varphi(u_1) = \varepsilon u_2$.
  For $v =  \bar \varepsilon \varphi(v_1)$
   \[\sprod{u_2}{v_2} = \sprod{u_1}{v_1} = \sprod{\varepsilon u_2}{\bar \varepsilon^{-1} v} = \sprod{u_2} v.\]
   By \Cref{dyadic:rot1} there exists $T' \in X(L)$ such that $T'(v) = v_2$ and $T'(u_2) = u_2$.
   Hence $T'(\varphi(u_1)) = \varepsilon u_2$ and $T'(\varphi(v_1)) = \bar\varepsilon^{\,-1}v_2$.
  Set $H_2 = \O u_2 \oplus \O v_2$.
  The claim follows by observing that the map $\sigma \colon H_2 \to H_2$ with $u_2 \mapsto \varepsilon^{-1}u_2$ and $v_2 \mapsto\bar\varepsilon v_2$ is an element of $U(H_2)$, which is equal to $S(H_2)$ by \Cref{dyadic:hyperbolic-plane}.
\end{proof}

\subsection{Jordan splittings}
The previous section shows that we may restrict ourselves to lattices which do not split a hyperbolic plane.
In this subsection we use the results of Jacobowitz to give criteria for a lattice to split a hyperbolic plane
and to control the invariants in a Jordan splitting,
which play an important role in the induction steps.

\begin{lemma}\label{dyadic:n<=k}
 Let $L = P \perp N$ be a hermitian lattice with $\rk P=2$, $\scale(P) = \FP^i$, $\norm(P)=\fp^k$,
 $\scale(N)=\FP^j$, $\norm(P)=\fp^n$ and $i<j$.
 If $n\leq k$, then $L$ splits a $\FP^i$-modular hyperbolic plane.
 \end{lemma}
\begin{proof}
 We may assume that $N$ is modular.

 First assume that $P \cong H(i,k)$ where
 $i+e \geq  2k \geq i$.
 To prove that
 \[L=H(i,k) \perp N \cong H(i) \perp N=L',\]
 we check the isometry conditions given in \Cref{dyadic:classification}:
 Note that $i=2k$ gives the contradiction $i<j\leq 2n \leq 2k=i$.
 Hence $i \neq 2k$, that is, $H(i,k)$ is subnormal.
  We conclude that both Jordan splittings have (1) the same Jordan type, (2) $\det(L) = \det(L')$ and (4) the Jordan constituents have the same determinant.
 For (3) we note that
 \[\norm(L^{\FP^i})=\norm(L)=\fp^n=\norm(L')=\norm(L'^{\,\FP^i}).\]
 Further $N \subseteq L^{\scale(N)}\subseteq L$
 yields
 $\norm(L^{\scale(N)})=\fp^n$
 and similarly $\norm(L'^{\scale(N)})=\fp^n$.

 Now assume that $P \cong A(i,k)$ where
 $i \leq 2k < i+e$.
 Since
 \[i <j \leq 2n \leq 2k < i+e < j+e,\]
 we have $\scale(N)\FD \subsetneq \norm(N)\O$.
 Hence we can find a $\FP^j$-modular lattice $N'$ of norm
 $\fp^n$ with
  \[\det(N') \det(H(i)) = \det(N) \det(A(i,k)).\]
 Now we show that
 \[L = A(i,k) \perp N \cong H(i) \perp N' = L'.\]
 As before, conditions (1)--(3) follow. It remains to check (4) which amounts to
  \[\det(A(i,k))/{\det(H(i))} \cong 1 \mod \fp^{2n-i}.\]
  Indeed this is true as $\det(A(i,k))/{\det(H(i))}-1 = u_0 \in \fp^{e-1} \subseteq \fp^{2k-i}\subseteq \fp^{2n-i}$.
\end{proof}

\begin{lemma}\label{dyadic:rels}
Suppose that $i<j$, $j-i \leq n - k$.
Let $M$ be a $\FP^i$-modular lattice of norm $\fp^k$.
Then the following holds:
\begin{enumerate}
 \item[(a)] We have
\[M \perp H(j,n) \cong M \perp H(j).\]
\item[(b)] If $i+e-2k > j - i$, then there is some lattice $M'$ with
\[M \perp A(j,n) \cong M' \perp H(j)\]

\item[(c)] Put $n' = j - i + k$, then
\[M \perp A(j,n) \cong M \perp A(j,n'),\]
\[M \perp H(j,n) \cong M \perp H(j,n').\]
\end{enumerate}
\end{lemma}
\begin{proof}
If $M$ is normal, then $M$ is a sum of rank one lattices and we can flip the norm class of the determinant of a single rank one sublattice to obtain $M'$.
Otherwise, by \Cref{dyadic:modular}, $M=P\perp H(i)^r$ with $P$ equal to $A(i,k)$ or $H(i, k)$. Since we assume $2k < i+e$, we are free to set $M' = A(i,k) \perp H(i)^r$ or $M'= H(i,k) \perp H(i)^r$.
We check conditions (1)--(4) of Theorem~\ref{dyadic:classification}:

(1) All decompositions have the same ranks and scales.
Suppose that $j=2n$, then $i\leq 2k <2n=j$ implies that
$j-i \geq 2(n-k)$. This contradicts $n-k \geq j-i>0$.
Thus $j< 2n$, that is, the second summand $H(j,n)$ or $A(j,n)$
is subnormal. Hence all decompositions have the same Jordan type.

(2) By construction the lattices have the same total determinant.

(3) Let $L=L_1 \perp L_2$ be the left hand side and $L'=L_1' \perp L_2'$ the right hand side in one of the decompositions and $\fp^{ n'} = \norm({L_2'})$.
It follows that
\[\norm_{1}= \norm(L) =\fp^k = \norm(L') = \norm'_{1},\]
\[\norm_{2}= \norm(L^{\FP^{j}}) = \fp^{k+j-i} + \fp^n= \fp^{k+j-i},\]
\[\norm'_{2}= \norm (L'^{\,\FP^{j}}) = \fp^{k+j-i} + \fp^{n'}.\]
Then $\norm_{2} = \norm'_{2}$ if and only if $j-i \leq n' -k$.

(4) This is true for (a) and (c).
For (b) we have to check that
  \[\det(M)/\det(M')=(1+u_0) \cong 1 \mod \mathfrak n_{1}\mathfrak n_{2}\fp^{-i}\]
which is equivalent to $e - 1 \geq 2k+j-2i$, that is,
$i+e -2k > j - i$.
\end{proof}

\begin{lemma}\label{dyadic:n-k=j-i}
   Let $P$ be a $\FP^i$-modular plane of norm $\fp^k$ and
   $L =P \perp M$ with $\scale(M)=\FP^j$, $\norm(M)=\fp^n$.
  If $0<j - i \leq n - k$, then $L \cong  P \perp M'$
  with $\scale(M')=\FP^j$ and $\norm(M')=\fp^{j-i+k}$.
\end{lemma}
\begin{proof}
  Let $M = \bigperp_{l= 1}^t M_l$ be a Jordan splitting. Suppose that $\rk M_1$ is odd. Then $M_1$ is subnormal and therefore $2n = j$. It follows that $j-i = 2n -i \geq 2(n-k) \geq 2(j-i)$, which contradicts $j-i>0$.
 Thus $\rk M_1 = 2 + 2r$ is even with $r \geq 0$ and
  $M_1$ is isomorphic to one of $H(j,n) \perp H(j)^r$ or $A(j,n) \perp H(j)^r$.
  By \Cref{dyadic:rels} (c) we find an isometry $P\perp M_1 \cong P \perp M_1'$
  with $\scale(M_1')=\scale(M_1)=\FP^j$ and $\norm(M_1')=\fp^{k+j-i}$.
\end{proof}

\subsection{The normal case}
We now consider the case that $L$ is normal and does not split a hyperbolic plane. Then $L = N \perp M$ or $L = P \perp M$ with
$\scale(M) \subseteq \FP \scale(N)$ and a line $N$ or normal plane $P$. Note that a normal plane $P=H(i,i/2)$ splits into two lines.
We shall prove that $U(L) = S(L) U(M)$.

\begin{lemma}\label{dyadic:reduce}
  Let $L = \O x \perp M$ be a hermitian lattice and
  $\varphi \in U(L)$ with $\varphi(x) = \alpha x + t$, where $\alpha \in \O$, $t \in M$. Then there exists $T \in S(L)$ such that $T(\varphi(x)) = \alpha' x + t$ and $v_\FP(1 - \alpha') \leq e-1$.
\end{lemma}

\begin{proof}
 Suppose that $\nu_\FP(1-\alpha)\geq e$.
 Recall that $\rho \in E$ is an element with $\Tr(\rho)=1$ and $\nu_\FP(\rho)=1-e$.
 Set $\sigma = \sprodq{x}{x}\rho$.
 Then
 \[\sprod{L}{x}\sigma^{-1}=\sprodq{x}{x}\sigma^{-1}\O=\rho^{-1}\O\subseteq \O\]
 and hence $S_{x,\sigma}\in S(L)$.
 We have
 \[S_{x,\sigma}(\varphi(x)) = \alpha ' x + t\]
 with $\alpha ' = \alpha (1-\sprodq{x}{x}\sigma^{-1})=\alpha(1-\rho^{-1})$.
 Then
 $\alpha' -1 \equiv -\rho^{-1} \mod (\alpha-1)\O$ and since $(\alpha-1)\O
 \subseteq \FP^e$ it follows that $v_\FP(1 - \alpha') = -v_\FP(\rho) = e-1$.
 This finishes the proof.
 \end{proof}

\begin{lemma}\label{dyadic:normal-rk1}
  Let $L = \O x \perp M$ be a hermitian lattice with $\scale(M) \subseteq \FP\sprodq{x}{x}$.
  If $L$ does not split a hyperbolic plane, then $U(L) = S(L)U(M)$.
\end{lemma}

\begin{proof}
  The statement is true for $M=0$. So assume $M \neq 0$.
  Let $\scale(\O x)=\FP^i$, $\scale(M)=\FP^j$, $\norm(\O x)=:\fp^k$ and $\fp^n :=\norm(M)$.
  Since $\O x$ is normal, we have $k < n$.
  Consider $\varphi \in U(L)$ and write $\varphi(x) = \alpha x + z$ with $\alpha \in \O$ and $z \in M$.
  Since
  \[\sprodq{x}{x}=\sprodq{\varphi(x)}{\varphi(x)}\equiv \alpha \bar \alpha\sprodq{x}{x} \mod \fp^n\]
  and $\fp^n\subseteq \fp^{k+1}$, we obtain $v := \nu_\FP(\alpha -1) \geq 1$ by \Cref{dyadic:goingdown}.

We have
\[ \sprod L {x - \varphi(x)}=(1-\bar\alpha)\FP^i+\sprod{M}{z} \subseteq \FP^{i+v}+\FP^j,\]
\[\sprod x {x - \varphi(x)}\O = \FP^{i+v},\]
and if $j-i \geq v$, it follows from \Cref{sym:act} that there exists $S \in S(L)$ with $S(x) = \varphi (x)$.

Let $j-i<v$.
By \Cref{dyadic:reduce}, we can assume that $v\leq e-1$. This settles the case $j-i\geq e-1$.

We continue with $j-i < e-1$.
If $M$ is normal, then $2n = j$, $2k=i$ and $j-i=2(n-k)>n-k$.
Now let $M$ be subnormal, that is, $j<2n$. In particular $M = P \perp M'$ for some $\FP^j$-modular plane $P$ of norm $\fp^n$. Note that $j+e>2n$ since $P \ncong H(j)$.
Suppose that $j-i \leq n-k$. If $P \cong H(j,n)$, then \Cref{dyadic:rels} (a) implies that $L$ splits a hyperbolic plane.
Otherwise $P\cong A(j,n)$. Since $j-i < e-1=i+e-2k - 1$ (using $i=2k$), we can apply \Cref{dyadic:rels} (b) to get a hyperbolic plane once again.
In both cases this is a contradiction to our assumption on $L$.
Thus in either case---$M$ normal or subnormal---we may assume that $n-k< j - i$.

Let $y \in M$ with $\sprodq{y}{y} = \epsilon p^{n-k} \sprodq{x}{x}$ for some $\epsilon \in \o^\times$.
By \Cref{dyadic:norm-mod}, after rescaling $y$ with a unit, we may assume that $1+\epsilon \equiv 0 \mod \fp^{e-1}$.
For $s = \pi^{n-k}x + y$ we have
\[\sprodq{s}{s} = \sprodq{x}{x}p^{n-k}(1+\epsilon)\]
and further
\[\sprod{L}{s} \subseteq \FP^{i+n-k} + \FP^j=\FP^{i+n-k}.\]
Set $\sigma' = \sprodq{s}{s}\rho$. Since $\nu_\FP(\sigma') \geq 2n + e - 1 \geq j > i+n-k$,
we can find $\sigma = \sigma' + \omega$ with $\Tr(\omega)=0$ and
\[i+n-k-\nu_\FP(\sigma)=n+k - \nu_\FP(\sigma) \in \{0,1\}.\]
Then $S_{s,\sigma} \in S(L)$ and
$S_{s,\sigma}(\alpha x + z) = \alpha'x + z'$ with $z' \in M$ and 
\begin{eqnarray*}
\alpha' &=& (\alpha- \alpha p^{n-k}\sprodq{x}{x}\sigma^{-1}-\sprod{z}{y}\sigma^{-1}\pi^{n-k})\\
        &\equiv& 1- (p^{n-k}\sprodq{x}{x} + \sprod{z}{y}\pi^{n-k})\sigma^{-1} \mod \FP^{j-i},
\end{eqnarray*}
where we used that $j-i<v$.
Using $2k=i$ and $n-k< j-i$ we have
\[\nu_\FP(p^{n-k}\sprodq{x}{x})=2n=n+k +(n-k) < n+k + j-i = j + n-k\leq \nu_\FP(\sprod{z}{y}\pi^{n-k}).\]
Thus
\[\nu_\FP(\alpha'-1)=2n - \nu_\FP(\sigma) = (n-k) + (n+k -\nu_\FP(\sigma)) \leq n-k+1 \leq j-i.\]
By the first part we find $S \in S(L)$ with $S(x) = S_{s,\sigma}(\varphi(x))$.
\end{proof}

\begin{lemma}\label{dyadic:normal-rk2}
  Let $L = \O x \perp \O y \perp M$ be a hermitian lattice with $\sprodq{x}{x}\O = \sprodq{y}{y}\O=\fp^k\O=\FP^i$ and $\scale(M) =\FP^{j}\subsetneq \FP^{i}$.
  Then $U(L) = S(L)U(\O y \perp M)$.
\end{lemma}
\begin{proof}
  Let $\varphi \in U(L)$ and write $\varphi(x) = \alpha x + \beta y + m$ for $\alpha,\beta \in \O$ and $m \in M$.
We have
\[\sprod{x}{x-\varphi(x)}=(1-\bar \alpha)\sprodq{x}{x}\]
and
\[\sprod{L}{x-\varphi(x)} = (1-\bar\alpha)\sprodq{x}{x} \O +\beta \FP^i + \sprod{M}{m}.\]
If $1-\alpha \in \O^\times$, then we can find a symmetry mapping $x$ to $\varphi(x)$. Otherwise $1 \equiv \alpha \mod \FP$ and
\[\sprodq{x}{x} = \sprodq{\varphi(x)}{\varphi(x)} \equiv \sprodq{x}{x} + \beta\bar\beta\sprodq{y}{y} \mod \FP^{k+1},\]
implies $\beta \bar \beta \in \fp$. In particular $\beta \in \FP$. If $\nu_\FP(1-\alpha)=1$, then
$\sprod{L}{x-\varphi(x)}=\FP^{i+1}$ and we can find a symmetry mapping $x$ to $\varphi(x)$.

Suppose that $1-\alpha \in \FP^2$.
By \Cref{dyadic:norm-mod} we find $\epsilon \in \O^\times$ such that $s= x + \epsilon y$ satisfies $\sprodq{s}{s}=\sprodq{x}{x}+\epsilon \bar \epsilon \sprodq{y}{y} \in \fp^{k+e-1}$.
Let $\sigma = \sprodq{s}{s}\rho+\omega$ with $\Tr(\rho)=1$, $\Tr(\omega)=0$ and
$2k - \nu_\FP(\sigma) \in \{0,1\}$. This is possible since
$\nu_\FP(\sprodq{s}{s}\rho) \geq 2k+e-1\geq 2k$.
Then $S_{s,\sigma} \in S(L)$.
We have $S_{s,\sigma}(\varphi(x)) =\alpha'x+\beta' y + m$ with $\beta' \in \O$ and
\begin{eqnarray*}
  1-\alpha' &=&1-\alpha + \sprod{\alpha x + \beta y}{x+\epsilon y}\sigma^{-1}\\
            &=&1-\alpha +(\alpha \sprodq{x}{x} + \beta \bar \epsilon \sprodq{y}{y})
\sigma^{-1}\\
&\equiv& (\sprodq{x}{x} + \beta \bar \epsilon\sprodq{y}{y})\sigma^{-1} \mod \FP^2.
\end{eqnarray*}
As further $\beta \in \FP$, we have $\nu_\FP(1-\alpha')=\nu_\FP(\sprodq{x}{x}\sigma^{-1})=2k-\nu_\FP(\sigma)\leq 1$.
By the first part we can find a symmetry mapping $x$ to $S_{s,\sigma}(\varphi(x))$.
\end{proof}

\subsection{The subnormal case}
The goal of this subsection is to prove the following result.

\begin{proposition}\label{dyadic:subnormal}
  Let $L = P \perp M$ be a hermitian lattice with $P$ a $\FP^i$-modular plane with $\norm(P)=\fp^k$,
 $\scale(M)\subseteq \FP^{i+1}$ and $i< 2k$. Assume that $L$ does not split a hyperbolic plane. Then $U(L)=S(L)U(M)$.
\end{proposition}

Before we can give the proof, we fix the setup and give a series of lemmas.
Throughout this subsection $L = P \perp M$ with $P$ a $\FP^i$-modular plane of norm $\fp^k$ and $M$ of scale $\FP^j$ and norm $\fp^n$. We assume that $L$ does not split a hyperbolic plane, $1\leq n-k \leq j-i$ and $P$ is subnormal, in particular $i<2k<i+e$ and $j\leq 2n < j+e$.
Let $u,v \in P$ with
\[\sprodq{u}{u}= p^k, \quad  \sprod{u}{v} = \pi^i \quad \mbox{ and }  \quad \sprodq{v}{v} \in \fp^{i-k+e-1}.\]
Set $v' = u -  \pi^i p^{k-i}v$ and note that $(\O u)^\perp = \O v' \perp M$. A quick calculation shows that the element $v'$ satisfies
  \begin{equation}\label{dyadic:defv-prime}
\sprod{u}{v'}=0, \quad
  \sprodq{v'}{v'}\equiv -  p^k \mod \fp^{k+e-1},\quad \sprod{v'}{v}\equiv \pi^i \mod \FP^{i+2(e-1)}.
  \end{equation}
  We assume that there exists $x \in M$ with $\sprodq{x}{x}=p^n$ and $\sprod L x = \scale(M)$ (if $M=0$, we set $n = \infty$ and adopt the convention $p^\infty=0$).
To justify this assumption let $a \in \o^\times$ be a unit. Note that $\prescript{a}{}{P}\cong P$ and changing $L$ to $\prescript{a}{}{L}$ does not affect norms, scales or being hyperbolic and $U(L)=U(\prescript{a}{}{L})$, $S(L) = S(\prescript{a}{}{L})$.

\begin{lemma}\label{dyadic:u}
 Suppose that either
 \begin{enumerate}
  \item[(a)] $2k+2 \leq i+e$ or
  \item[(b)] $2k+1 = i+e$ and $\sprodq{v}{v}\o=\fp^{i-k+e-1}$.
 \end{enumerate}
 Let $\varphi \in U(L)$.
 Then there exists $\psi \in S(L)$ with $\varphi(u) = \psi(u)$.
\end{lemma}
\begin{proof}
  Write $\varphi(u) = \alpha u + \beta v + w$ for $w\in M$.

 (a) Then
 \[p^k = \sprodq{\varphi(u)}{\varphi(u)} = \alpha \bar \alpha  p^k + \Tr(\alpha \bar \beta \pi^i) + \beta\bar\beta\sprodq{v}{v}+ \sprodq{w}{w}.\]
 Since by our assumptions $k+1 \leq \lfloor (i+e)/2 \rfloor$ and $\sprodq{v}{v} \in \fp^{k+1}$, we have
  \[ p^k \equiv \alpha \bar  \alpha p^k \pmod{ \Tr(\FP^i)+\fp^{k+1}+\fp^n=\fp^{k+1}}.\]
  This gives $\alpha \bar \alpha \equiv 1 \mod \fp$ and thus $\alpha -1 \equiv 0 \mod \FP$ by \Cref{dyadic:goingdown}.

  If $\nu_\FP(\beta)> 1$, then we consider $S_{v,\sigma}$ with
  $\sigma=\sprodq{v}{v}\rho+\omega$ where we choose $\omega \in E$ skew and such that $i - 1 \leq \nu_\FP(\sigma) \leq i$. This is possible, since
  \[\nu_\FP(\sprodq{v}{v}\rho)\geq 2(i-k)+e-1 = i + (i+e - 2k -1) \geq i.\]
As we now have
\[S_{v,\sigma}(\varphi(u)) = \alpha u + (\beta-(\alpha \pi^i+\beta \sprodq{v}{v})\sigma^{-1})v+w,\]
  after replacing $\varphi$ by $S_{v,\sigma}\circ \varphi$ we can assume that $\nu=\nu_\FP(\beta)\leq 1$.

  Then,
  using that $i<2k$ and $\alpha \equiv 1 \mod \FP$, we arrive at
  \[ \sprod u {u - \varphi(u)}\O =( p^k (1-\bar \alpha) -\bar \beta \pi^i) \O = \pi^{i+\nu} \O = \sprod{u - \varphi(u)}{L}. \]
  Hence by \Cref{sym:act} there exists a symmetry $S \in U(L)$ with $S(u) = \varphi(u)$.

  (b) We have $k = \lfloor (i+e)/2 \rfloor$ and $\sprodq{v}{v}\o = \fp^{i-k+e-1}$.
  If $\beta \in \O^\times$, then, since $i<2k$,
  \[ \sprod u {u - \varphi(u)}\O =( p^k (1-\bar \alpha) -\bar \beta \pi^i) \O = \pi^{i} \O = \sprod{u - \varphi(u)}{ L}, \]
  and by \Cref{sym:act} we can map $\varphi(u)$ to $u$ via a symmetry.
  Otherwise let $\beta \in \FP$. As before one can show that this implies $\alpha \in \O^\times$.
  For the symmetry
  $S_{v,\sigma}$ with $\sigma=\sprodq{v}{v}\rho$ we have
  \[\nu_\FP(\sigma)=2(i-k)+e-1 = i.\]
  Then, using that $i<2k$ and $\alpha \in \O^\times$, we obtain
  \[S_{v,\sigma}(\varphi(u)) = \alpha u + (\beta-(\alpha \pi^i+\beta \sprodq{v}{v})\sigma^{-1})v+w=\alpha u + \beta'v+w\]
  with $\beta' \in \O^\times$.
\end{proof}

\begin{lemma}\label{dyadic:aik-norm}
 Let $2n+1 = j+e$.
 If $m \in A(j,n)$ is primitive, then $\sprodq{m}{m}\o=\fp^n$.
\end{lemma}
\begin{proof}
See \cite[9.2 b)]{Jacobowitz1962}.
 \end{proof}

  \begin{lemma}\label{dyadic:mm}
  Let $j-i = n-k$.
  Consider $\varphi \in U(L)$ with $\varphi(u) = u$ and
  $\varphi(v) = \gamma u + \delta v + m$ with $\delta, \gamma \in \O$ and $m \in M$.
  If $\gamma \in \FP^{j-i+1}$ and $\delta \equiv 1 \mod \FP$,
  then $\sprodq{m}{m} \in \fp^{n+1}$.
  \end{lemma}
 \begin{proof}
 Note that $2n<j+e$. Then $2n+2 \leq j+e+1$, so $n+1 \leq \lfloor (j+e+1)/2\rfloor$. Further $j-i=n-k$ gives $i-k+e>n$.
  Now
  \[0=\sprodq{\varphi(v)}{\varphi(v)} - \sprodq{v}{v} = \gamma \bar \gamma p^k + \Tr(\gamma \bar \delta \pi^i)+ (\delta\bar\delta-1) \sprodq{v}{v}+\sprodq{m}{m}\]
  yields $0\equiv \sprodq{m}{m} \mod \fp^{n+1}$.
  \end{proof}

\begin{lemma}\label{dyadic:eierlegendewollmichsau}
 Let $\gamma \in \FP^{j-i+1}$, $\delta \equiv 1 \mod \FP$ and $m \in M$.
 Then there are $s\in L$ and $\sigma \in E$
 such that the symmetry $S_{s,\sigma}$ fixes $u$, preserves $L$ and further
 $S_{s,\sigma}(\gamma u + \delta v + m)= \gamma'u + \delta'v + m'$
with $\gamma', \delta' \in \O$, $m' \in M$ and
\[\gamma ' \equiv  - (\bar\pi^{n-k+i} + \sprod{m}{x})\sigma^{-1}\pi^{n-k} \mod \FP^{n-k+1+q}.
\]
where $n-k+i - \nu_\FP(\sigma) =q\in \{0,1\}$.
If further $n-k=j-i$ and $\lfloor (j+e)/2 \rfloor>n$, then $q=0$.
\end{lemma}
\begin{proof}
  Let $h = \lfloor (j+e)/{2} \rfloor$.
  If $h>n$, then by \Cref{dyadic:norm-mod} we can find
  $\epsilon \in \O^\times$ with $\epsilon \bar \epsilon \equiv 1 - p^{h-n} \mod \fp^{e-1}$. By \Cref{dyadic:goingdown} we have $\epsilon \equiv 1 \mod \FP$. If $h=n$, then $j \not \equiv e \mod 2$ and we set $\epsilon =1$.
  Recall the properties of $v'$ in \Cref{dyadic:defv-prime}:
  \[\sprod{u}{v'}=0, \quad
  \sprodq{v'}{v'}\equiv -  p^k \mod \fp^{k+e-1},\quad \sprod{v'}{v}\equiv \pi^i \mod \FP^{i+2(e-1)}\]
  and that  $\sprod{L}{x} = \FP^j$. Now we set
  \[s = \epsilon \pi^{n-k}v' + x.\]
  Then $\sprod{u}{s}=0$, $\sprod{L}{s}=\FP^{n-k+i}+\FP^j = \FP^{n-k+i}$. Further
  \[
  \sprodq{s}{s} = \epsilon \bar \epsilon p^{n-k}\sprodq{v'}{v'} + p^n
   \equiv  -\epsilon \bar \epsilon p^{n} + p^n
   \equiv
    \begin{cases}
     0, & \mbox{if } n=h,\\
     p^h, &  \mbox{else,}
   \end{cases}
\mod \fp^{n+e-1}.
\]

Note that $e\geq 2$ and $2n\geq j$; hence
\[2(n+e-1 - h) = (2n - j) + (e-2) +  \begin{cases}
                                            0 &  \mbox{ if }  j\equiv e \mod 2 \\
                                            1 &\mbox{ if }j \not \equiv e \mod 2.
                                            \end{cases}\]
Thus $ n+e-1 \geq h + 1$, except if $2n=j$ and $e=2$.
But in any case $n+e - 1 \geq h$.
This gives that $\nu_{\fp}(\sprodq{s}{s}) \geq h$.
Recall that $\rho \in E^\times$ is such that $\Tr(\rho)=1$ and $\nu_\FP(\rho) = 1 - e$. Hence $\nu_\FP(\sprodq{s}{s}\rho)\geq 2h + 1-e  \geq j$.

  Since $n-k+i \leq j$, we can find a skew element $\omega \in E$ with
  $\sigma=\sprodq{s}{s}\rho + \omega$ such that
 $n-k+i -\nu_\FP(\sigma)=q' \in \{0,1\}$.
 We show $q' = q$:  If $n - k + i < j$, there is nothing to show. So suppose that $n-k = j-i$. Then $q' = j - \nu_\FP(\sigma)$.

 If $j \equiv e \mod 2$, then we can choose $\omega$ to be zero or of valuation $j$ so that $\sigma$ has valuation $j$ as well and $q'=0$.

 If $j \not \equiv e \mod 2$, then $2(n+e-1-h)\geq 1$. Hence
 $n+e-1 \geq h+1$. In particular the valuation of $\nu_\FP(\sprodq{s}{s}\rho) = 2h+1-e = j$. So we can take $\omega =0$ and then $q' =0$.

 In any case $\sprod{L}{s}\subseteq \sigma \O$, so that $S_{s,\sigma} \in S(L)$.
 We have
  \begin{eqnarray*}
   \gamma' & = & \gamma - \sprod{\gamma u + \delta v + m}{\epsilon \pi^{n-k}v' +x}\sigma^{-1}\epsilon \pi^{n-k}\\
   &=& \gamma - (\delta \bar \epsilon \bar\pi^{n-k}\sprod{v}{v'} + \sprod{m}{x})\sigma^{-1}\epsilon\pi^{n-k}\\
   &\equiv & \gamma - (\delta \bar \epsilon \bar\pi^{n-k+i} + \sprod{m}{x})\sigma^{-1}\epsilon\pi^{n-k} \mod \FP^{n-k+q+2(e-1)}
   \end{eqnarray*}
   Note that $\epsilon \equiv \delta \equiv 1 \mod \FP$ and $2(e-1) \geq 1$. Hence, using that $a\equiv b \mod \FP^g$ and $c \equiv 0 \mod \FP^f$ implies $ac \equiv bc \mod \FP^{f+g}$, we get
   \begin{eqnarray*}
   \phantom{\gamma'}&\equiv & \gamma - (\bar\pi^{n-k+i} +  \sprod{m}{x})\sigma^{-1}\pi^{n-k} \mod \FP^{n-k+1 +q}\phantom{blablalla}\\
   &\equiv & - (\bar\pi^{n-k+i} + \sprod{m}{x})\sigma^{-1}\pi^{n-k} \mod \FP^{n-k+1 +q}
  \end{eqnarray*}
  where the last congruence follows from the assumption $\gamma \in \FP^{j-i+1}$ and
  $n-k+1+q \leq j-i +1$.
\end{proof}

  \begin{lemma}\label{dyadic:mx}
  Let $M = A(j,n) \perp M'$ with $\scale(M')\subseteq \FP^{j+1}$ and $\norm(M')\subseteq \fp^{n+1}$.
  Let $\varphi \in U(L)$ be an isometry with $\varphi(u) =u$ and
  $\varphi(v) = \gamma u + \delta v + m$ with $\gamma, \delta \in \O$, $m \in M$
  and $\sprodq{m}{m} \in \fp^{n+1}$.
  Suppose that $j-i=n-k$ and $\left\lfloor \frac{j+e}{2}\right\rfloor >n > \lfloor \frac{i+e}{2} \rfloor$.
  Then
   \[\sprod{m}{x} \not \equiv \bar\pi^j \mod \FP^{j+1}.\]
  \end{lemma}

  \begin{proof}
  Recall that $v'$ is an element satisfying \Cref{dyadic:defv-prime} and that $x \in M$ is such that $\sprodq{x}{x}=p^n$.
  Let $y \in M$ with
  $\sprod{x}{y}=\pi^j$,
  $\sprodq{y}{y}=-u_0p^{j-n}$, so that $\O x \oplus \O y= A(j,n)$. Note that $\lfloor \frac{j+e}{2}\rfloor\geq n+1$
  implies that $\nu_\fp(-u_0p^{j-n})> n$.
  Since $\varphi(u)=u$ and $x \in (\O u)^\perp$, we have $\varphi(x) \in (\O u)^\perp=\O v' \perp M$.
  Hence we find $m',m'' \in M'$ and $b_1,a_{ij} \in \O$ with
  \[\begin{array}{ccccc}
    \varphi(v) &=& \gamma u + \delta v &+a_{11} x &+ a_{12}y+m'\\
    \varphi(x) &=&b_1 \pi^{n-k}v'&+ a_{21} x &+ a_{22}y+m''.\\
  \end{array}\]
  Then $\sprodq{m}{m}\equiv 0 \mod \fp^{n+1}$,
  implies that $a_{11} \equiv 0 \mod \FP$.
  Now $\sprod{u}{v-\varphi(v)}=0$ gives that $\delta \equiv 1 \mod \FP$.
  Then
  $0=\sprod{\varphi(v)}{\varphi(x)}\bar \pi^{-j}\equiv \bar b_1 +a_{12}\bar a_{21} \mod \FP$.
  Moreover we have
  \[1= \sprodq{\varphi(x)}{\varphi(x)}p^{-n}\equiv -b_1 \bar b_1  + a_{21}\bar a_{21} \equiv a_{21} \bar a_{21}(1-a_{12}\bar a_{12})\mod \fp \]
  Now $a_{12}\bar a_{12}\equiv 1 \mod \fp$ leads to a contradiction. In particular
  $a_{12} \not \equiv 1 \mod \FP$.
  As a consequence
  $\sprod{m}{x}\equiv \sprod{a_{12}y}{x}\equiv \bar \pi^j a_{12} \not \not \equiv \bar\pi^j \mod \FP^{j+1}.$
  \end{proof}

Note that the appearance of conditions (a-d) of the following lemma will be explained by the absence of a hyperbolic plane splitting $L$.
\begin{lemma}\label{dyadic:subnormal-plane}
  Recall that $L =  P \perp M$ with $\rk P=2$, $\scale(P)=\FP^i$
  $\norm(P)=\fp^k$, $P \not\cong H(i)$, $i<2k$,
  $\scale(M)=\FP^j\subseteq \FP^{i+1}$, $\norm(M)=\fp^n \subseteq \fp^{k+1}$.
  Suppose that one of the following is true:
  \begin{enumerate}
   \item[(a)] $M=0$,
   \item[(b)] $j - i > n - k$,
   \item[(c)] $j - i = n - k$, $M \cong A(j,n)\perp M'$, $\left\lfloor \frac{j+e}{2}\right\rfloor =n > \lfloor \frac{i+e}{2} \rfloor$,
   \item[(d)] $j - i = n - k$, $M \cong A(j,n)\perp M'$,  $\left\lfloor \frac{j+e}{2}\right\rfloor >n > \lfloor \frac{i+e}{2} \rfloor$,
  \end{enumerate}
 where $M'$ is a lattice with $\scale(M')\subseteq \FP^{j+1}$ and $\norm(M')\subseteq \fp^{n+1}$.
Then $U(L) = S(P)U(M)$.
\end{lemma}
\begin{proof}
   Recall that $x \in M$ is such that $\sprodq{x}{x} =  p^n$.
   Let $\varphi \in U(L)$ and recall that $u, v \in P$ satisfy
   \[\sprodq{u}{u}= p^k, \quad \sprod{u}{v}=\pi^i, \quad \sprodq{v}{v}\in \fp^{i-k+e-1}.\]
   Since $P\not\cong H(i)$, \Cref{dyadic:u} applies and provides $\psi \in S(L)$ with $\psi(u) = \varphi(u)$.
   We continue with $\varphi' = \psi^{-1} \circ \varphi$, which satisfies $\varphi'(u) = u$.

   Write $\varphi'(v) = \gamma u + \delta v + m$ for some $m \in M$.
   Since $u = \varphi'(u)$, we have
   \[-\bar\gamma p^k + (1-\bar \delta)\pi^i=\sprod{u}{v-\varphi'(v)}
  = 0.\]
  Hence $\nu_\FP(1-\bar \delta)=\nu_\FP(\gamma)+2k-i$ and so
  \[\sprod{v}{v-\varphi'(v)}\O=(-\bar \gamma \bar \pi^i+(1-\bar \delta)\sprodq{v}{v})\O=\bar \gamma \bar \pi^i\O.\]
  The symmetry $S_{s,\sigma}$ with $s=v - \varphi'(v)$ and $\sigma = \sprod{v}{v-\varphi'(v)}$ preserves $u$ and maps $v$ to $\varphi'(v)$.
  It preserves $L$ if $\sprod{L}{s}\sigma^{-1} \subseteq \O$, which is the case if
  \begin{equation*}\label{eqn:hikgood}
  \sprod{M}{m}\subseteq \gamma \pi^i\O.
  \end{equation*}
  For $M=0$ this is true and we are done.
  Otherwise we set $\FP^l = \sprod{M}{m}$.
  Then the condition amounts to $\nu_\FP(\gamma) \leq l - i$.
  We continue with $\nu_\FP(\gamma)>l-i\geq j - i$.
  Note that as we saw above, $\sprod{u}{v-\varphi'(v)}=0$ implies $\delta \equiv 1 \mod \FP$.

  Let $S_{s,\sigma}$ be as in \Cref{dyadic:eierlegendewollmichsau}.
  Then $S_{s,\sigma}(\varphi'(v))= \gamma'u + \delta'v + m'$ with $m' \in M$
  and
  $\gamma ' \equiv  - (\bar\pi^{n-k+i} + \sprod{m}{x})\sigma^{-1}\pi^{n-k} \mod \FP^{n-k+1+q}$.
  The proof is complete if we can show that $\nu_\FP(\gamma') \leq l' - i$ where $\sprod{M}{m'}=\FP^{l'}$.

  (b) If $n-k<j-i$, then
  \[\nu_\FP(\gamma')=n-k+i < j \leq \nu_\FP(\sprod{m}{x}).\]
  Hence $\nu_\FP(\gamma')=2(n-k)+i - \nu_\FP(\sigma) \leq n-k+1 \leq j - i$. This concludes case (b).

  (c) Since $\norm(M')\subseteq \fp^{n+1}$ and $\sprodq{m}{m}\in\fp^{n+1}$ by \Cref{dyadic:mm}, \Cref{dyadic:aik-norm} gives $m \in \FP M + M'$. In particular $l> j$. By the same reasoning for $m'$ we obtain that $l' > j$ as well. Then
  \[\nu_\FP(\bar \pi^{j})=j < \nu_\FP(\sprod{m}{x}).\]
  Hence $\nu_\FP(\gamma')=2(n-k)+i - \nu_\FP(\sigma) \leq n-k+1=j-i+1 \leq l' - i$.

  (d) By \Cref{dyadic:mx}
  $ \bar\pi^{j} - \sprod{m}{x}  \not\equiv 0 \mod \FP^{j+1}$.
  Hence the valuation of
    \[\gamma' \equiv  - (\bar\pi^j + \sprod{m}{x})\sigma^{-1}\pi^{n-k} \mod \FP^{n-k+1+q}\]
  is indeed given by $j+(n-k)-\nu_\FP(\sigma)=j-i+q =j-i$.
  \end{proof}

\begin{proof}[Proof of \Cref{dyadic:subnormal}]
 If $M=0$, then this is \Cref{dyadic:subnormal-plane} (a).
 We set $\FP^j = \scale(M)$, $\fp^n=\norm(M)$.
 Then \Cref{dyadic:n<=k} implies that $k<n$. If $j-i > n -k$ then, \Cref{dyadic:subnormal-plane} (b) settles the proposition. Otherwise $j-i\leq n - k$ and \Cref{dyadic:n-k=j-i} applies. It states that we may alter the Jordan decomposition in such a way that $j-i=n-k$. In particular $M$ is subnormal.
 Write $M = M_1 \perp M_2$ with $M_1$ being $\FP^j$-modular and $\scale(M_2) \subsetneq \FP^j$. Since $L$ does not split a hyperbolic plane, $\rk M_1 \leq 2$ and since $M$ is subnormal, $\rk M_1 \geq 2$. Thus $M_1$ is a plane.
 If $M_1 \cong H(j,n)$, then \Cref{dyadic:rels} (a) implies that $L$ splits a hyberbolic plane.
 Thus
 $M_1 \cong A(j,n)$. If $i+e-2k > j-i$, then $L$ splits a hyperbolic plane as well by \Cref{dyadic:rels}~(b). Thus
 $i+e-2k \leq j-i$, then $i+e \leq n+k < 2n$, hence
 \[\left\lfloor \frac{i+e}{2}\right\rfloor < n \leq \left\lfloor \frac{j+e}{2}\right\rfloor.\]
 If $\norm(M_2)\supseteq \fp^n$, then \Cref{dyadic:n<=k} implies that we can split a hyperbolic plane. Thus $\norm(M_2)\subseteq \fp^{n+1}$.
 Now \Cref{dyadic:subnormal-plane} (c) or (d) apply and conclude the proof.
\end{proof}

\subsection{Proof of the generation theorem}
\begin{theorem}
  Let $E/K$ be a ramified quadratic extension of a local dyadic field of characteristic zero and $L$ a hermitian lattice, then the unitary group $U(L)$ is generated by symmetries and rescaled Eichler isometries, that is,
$U(L)=X(L)$.
\end{theorem}
\begin{proof}
The proof is by induction on the rank of $L$.
If $L$ is of rank $0$, then the statement is trivially true.
Suppose the theorem is proven for all lattices of rank strictly smaller than $\rk L$.
If $L$ splits a hyperbolic plane $P$, then \Cref{dyadic:hyperbolic} gives $U(L)=X(L)U(P^\perp)$. By induction $U(P^\perp)=X(P^\perp)$ and we are done.

Otherwise $L$ does not split a hyperbolic plane.
Then $L = P \perp M$ with a $\scale(L)$-modular line or plane $P$ and $\scale(M) \subsetneq \scale(L)$.
Now \Cref{dyadic:normal-rk1} and \Cref{dyadic:normal-rk2},  or \Cref{dyadic:subnormal} provide that
$U(L)=S(P)U(P^\perp)$ and the induction proceeds.
\end{proof}

The proof shows that if $L$ does not split a hyperbolic plane, then $U(L)=S(L)$.

\appendix

\section{Generation of unitary groups by symmetries}

It was shown in the main part of the paper that for a hermitian lattice over a two-dimensional étale algebra over a non-Archimedean field of characteristic zero, the unitary group is generated by symmetries and rescaled Eichler isometries.
  Moreover, except in the ramified dyadic case, it was shown that symmetries suffice.
  The aim of the appendix is to analyze the situation in the ramified dyadic case.
  More precisely, the following is shown.

\begin{theorem}\label{symmetriesgenerators}
  Let $E/K$ be a ramified quadratic extension of dyadic local fields of characteristic zero.
  Then the unitary group of every hermitian lattice over $E$ is generated by
  symmetries if and only if the residue field of $E$ is not $\FF_2$.
\end{theorem}

In the following we let $L$ be a hermitian lattice over the ramified quadratic extension $E/K$ of dyadic local fields. Let $u,v\in L$ be isotropic with $\sprod{u}{v}\O=\pO^i$, and $L=(\mc Ou \oplus \mc Ov)\perp M$. Consider a rescaled Eichler isometry $E_w^\mu\in U(L)$ with respect to this hyperbolic plane.


The following three results are \cite[5.4]{Hayakawa1968} and the omitted proofs are straightforward computations.

\begin{lemma}\label{compeich}
If $E^{\mu}_{w}$ and $E^{\mu'}_{w'}$ are two rescaled Eichler isometries of $L$ with respect to the same hyperbolic plane, then $E_{w}^{\mu}\circ E_{w'}^{\mu'}=E_{w+w'}^{\mu+\mu'-\sprod{w'}{w}/\sprod{u}{v}}$.
\end{lemma}

\begin{lemma}\label{changemu}
  Assume that $S_{u, w}$ is a symmetry and $E_{w}^\mu$ a rescaled Eichler isometry of $L$.
  If $\omega\in E$ is a skew element, then $S_{u,\omega}\circ E_{w}^\mu=E_w^{\mu-\sprod{v}{u}/\omega}$.
\end{lemma}

\begin{lemma}\label{unitirr}
  If $E_w^\mu$ and $E_w^{\mu'}$ are two rescaled Eichler isometries of $L$ with respect to the same hyperbolic plane, $E_w^{\mu'}\in S(L)\circ E_w^\mu$.
\end{lemma}

\begin{proof}
  Since $\Tr(\mu\sprod{u}{v})=-\sprodq{w}{w} =\Tr(\mu'\sprod{u}{v})$, we have $\mu'-\mu=\omega{\sprod{u}{v}}^{-1}$ for a skew element $\omega \in E$.
  Since the claim trivially holds if $\mu = \mu'$, we may assume that $\mu \neq \mu'$ and hence $\omega \neq 0$.
  Note that also $\omega'={\Nr(\sprod{v}{u})}{\omega}^{-1}$ is a skew element. The claim follows from \Cref{changemu}.
\end{proof}

Note that $S(L)$ is a normal subgroup of $U(L)$. For if $S_{s, \sigma}$ is a symmetry of $L$ and $f \in U(L)$ then one easily verifies $f\circ S_{s,\sigma}=S_{f(s),\sigma}\circ f$. By \Cref{unitirr}, two rescaled Eichler isometries $E_w^{\mu}$ and $E_w^{\mu'}$ of $L$ have the same class in $U(L)/S(L)$, which we will denote by $E_w\circ S(L)$. In particular, we say that $E_w\circ S(L)$ \textit{exists} if there is an element $\mu\in\O$ such that $E_w^\mu$ is a rescaled Eichler isometry of $L$.

\begin{lemma}\label{resceichlerexists}
  Let $w \in M$ with $\sprod{w}{L}\subseteq \mathfrak{P}^i=\sprod{u}{v}\O$. Then
  $E_w\circ S(L)$ exists if and only if $\sprodq{w}{w}\in\Tr(\mathfrak P^i)$.
\end{lemma}
\begin{proof}
If $E_w\circ S(L)$ exists, say $E_w^\mu$ is a rescaled Eichler isometry of $L$, we have $-\sprodq{w}{w}=\Tr(\mu\sprod{u}{v})\in \Tr(\mathfrak P^i)$.

Conversely assume that $\sprodq{w}{w}\in\Tr(\mathfrak P^i)$, say $-\sprodq{w} w=\Tr(\alpha)$ with $\alpha\in\mathfrak P^i$. Then the element $\mu={\alpha}{\sprod{u}{v}}^{-1}\in\O$ is such that $E^\mu_w$ is a rescaled Eichler isometry of $L$.
\end{proof}

The subsequent lemma generalizes \cite[5.1, 5.2, 5.7]{Hayakawa1968}. Recall from the introduction of Section 5 of the main paper that there exists $e\in\ZZ$, only depending on $E/K$, with the following properties:

\begin{itemize}
\item $\Tr(\mathfrak P^i)=\mathfrak p^{\floor{(i+e)/2}}$ for all $i\in\ZZ$,
\item if $e$ is odd, then there exists a skew prime $\eta\in\mc O$,
\item if $e$ is even, then there exists a skew unit $\eta\in\mc O$.
\end{itemize}

\begin{lemma}
  For a rescaled Eichler isometry $E_w^\mu$ of $L$ we have $E_w^\mu\in S(L)$ if one
  of the following holds:

  \begin{enumerate}
  \item[(a)] $\mu\in \O^\times$,
\item[(b)] $\mu\in\mathfrak P^2$ and $E_w^{\mu'}\in S(L)$ for all $\mu'\notin\mathfrak P^2$ with $-\sprodq{w}{w} =\Tr(\mu\sprod{u}{v})$,
  \item[(c)] $i\equiv e\mod 2$,\label{unitprime}
  \item[(d)] $\sprod{w}{w}\notin\mathfrak p\Tr(\mathfrak P^i)$,\label{wnorm}
  \item[(e)] $\sprod{w}{L}\ne\mathfrak P^i$.\label{wscale}
  \end{enumerate}
\end{lemma}

\begin{proof}
For $\mu\ne0$, define the symmetries $S_1=S_{\mu u+w,-\overline\mu\sprod{v}{u}}$, $S_2=S_{w,-\mu\sprod{u}{v}}$ of $L$. A quick calculation yields $E_{w}^{\mu}=S_1\circ S_2$. Thus $E_w^\mu\in S(L)$ if and only if $S_1\in U(L)$, if and only if $S_2\in U(L)$.\par
  (a): By definition 
$\sprod{\mu u+w}{L}=\sprod{u}{L}=\mathfrak P^i=-\overline\mu\sprod{v}{u}\mc O$, whence $S_1\in U(L)$.

(b),(c): Let $\omega \in E$ be a skew element, which will be specified below, and set $\mu'=\mu+\omega{\sprod{u}{v}}^{-1}$. We have $\Tr(\mu'\sprod{u}{v})=\Tr(\mu\sprod{u}{v}+\omega)=-\sprodq{w}{w}$.

If the condition of (c) is satisfied, by choosing $\omega\in\eta K^\times$ we can achieve $\omega{\sprod{u}{v}}^{-1} \in\O^\times$ for $\mu\notin\O^\times$ and $\omega{\sprod{u}{v}}^{-1}\in\mathfrak P^2$ for $\mu\in\O^\times$. Hence $\mu'\in\O^\times$ and $E_w^{\mu'}\in S(L)$ by (a).

Now assume that the condition of (c) is not satisfied. Again choosing $\omega\in\eta K^\times$, we can achieve only $\mu'\notin\mathfrak P^2$. Here assumption (b) yields $E_w^{\mu'}\in S(L)$.

  In either case, $E_w\circ S(L)=S(L)$, and so $E_w^\mu\in S(L)$, too.

  (d): Assume $\mu\in\mathfrak P$. Thus $\mu\sprod{u}{v}\in\mathfrak P^{i+1}$ and
  \[-\sprodq{w}{w}=\Tr(\mu\sprod{u}{v})\in\Tr(\mathfrak P^{i+1}).\]
  By (c) we may assume that $i+e$ is odd. Thus
  \[\sprodq{w}{w}\in\Tr(\mathfrak P^{i+1})=\mathfrak p^{\floor{(i+e+1)/2}}=\mathfrak p^{\floor{(i+e)/2} + 1}=\mathfrak p\Tr(\mathfrak P^i).\]
  This contradicts the assumption and so $\mu\notin\mathfrak P$. The claim follows by (a).\par
  (e): Since by definition $\sprod{w}L\subseteq \FP^i$ we obtain $\sprod{w}{L}\subseteq\mathfrak P^{i+1}$. Suppose $\mu\notin\mathfrak P^2$. Hence $\sprod{w}{L}\subseteq\mathfrak P^{i+1}\subseteq\mu\mathfrak P^i=-\mu\sprod{u}{v}\mc O$ and so $S_2\in U(L)$. Now for $\mu\in\mathfrak P^2$ the claim follows from (b).
\end{proof}

To finish the proof, we will need the additional assumption that the residue field contains more than two elements. We generalize the argument from \cite[5.6]{Hayakawa1968}.

\begin{lemma}\label{fieldsize}
  Assume that $\lvert \o/\fp \rvert \geq 4$. Then for any $\gamma,\delta\in\o$, there exists an element $\alpha\in\o^\times$ such that $\alpha+\gamma,\alpha+\delta\in\o^\times$.
\end{lemma}

\begin{proof}
  Since $\bar x+\bar \gamma = \bar 0$, $\bar x+ \bar \delta = \bar 0$, $\bar x = \bar 0$, are linear equations over $\o/\fp$, there are $\lvert \o/\fp \rvert -3$ elements of $\o/\fp$ which are no root of either one. Thus by assumption there exists such an element $\bar \alpha \in \o/\fp$. Thus the elements $\alpha,\alpha+\gamma,\alpha+\delta$ are units.
\end{proof}

\begin{theorem}\label{apptheorem2}
  Let $E/K$ be a ramified, quadratic dyadic extension. If the residue field of $E$ (and $K$) is not $\Fi_2$, then $X(L)=S(L)$.
\end{theorem}

\begin{proof}
  We consider a vector $w\in M$, where $L=(\mc Ou \oplus \mc Ov)\perp M$ as before, and show $E_w\circ S(L)=S(L)$. To this end we will construct $w' \in M$ with $\sprodq{w'}{w'}\o=\sprodq{w - w'}{w-w'}\o=\Tr(\mathfrak P^i)$.

  By \Cref{fieldsize} (with $\gamma=\delta=-1$) there exists an element $\zeta\in\o^\times$ such that $\zeta-1\in\o^\times$.
  By \Cref{wscale} (e) we may assume that $\sprod{w}{L}=\mathfrak P^i$  whence there is a $t\in M$ with $\Tr(\sprod{w}{t})\o=\Tr(\mathfrak P^i)$. By \Cref{wnorm} (d) we may assume $\sprodq{w}{w}\in\mathfrak p\Tr(\mathfrak P^i)$.

  First case: $\sprodq{t}{t}\notin\Tr(\mathfrak P^i)$. Thus $\Tr(\mathfrak P^i)=\sprodq t t\mathfrak p^j$ with $j>0$. By \Cref{unitprime} (c) we may assume that $i+e$ is odd and therefore
  \[\Tr(\pi^j\sprod{w}{t})\in\Tr(\mathfrak P^{i+j})\subsetneq\Tr(\mathfrak P^i).\]
  Thus setting $w'=w+\pi^jt$ yields
  \begin{align*}
    \sprodq{w'}{w'}&=\sprodq w w+\Norm(\pi^j)\sprodq t t+\Tr(\pi^j\sprod{w}{t}),\\
    \sprodq{w - w'}{w-w'}&=\Norm(\pi^j)\sprodq t t,
  \end{align*}
  which implies $\sprodq{w'}{w'}\o = \sprodq{w - w'}{w - w'}\o = \sprodq t t \fp^j = \Tr(\FP^i)$ as desired.

  Second case: $\sprodq t t\in\mathfrak p\Tr(\mathfrak P^i)$. Here $w'=\zeta w+t$ yields
  \begin{align*}
    \sprodq{w'}{w'}&=\zeta^2\sprodq w w+\sprodq t t+\zeta \Tr(\sprod{w}{t}),\\ 
    \sprodq{w - w'}{w-w'}&=(\zeta-1)^2\sprodq w w+\sprodq t t+(\zeta-1)\Tr(\sprod{w}{t}), 
  \end{align*}
  which implies $\sprodq{w'}{w'}\o = \sprodq{w - w'}{w - w'}\o = \Tr(\sprod w t)\o = \Tr(\FP^i)$.

  Third case: $\sprodq{t}{t}\o=\Tr(\mathfrak P^i)$. Using \Cref{fieldsize} again we obtain $\alpha\in\o^\times$ such that
  \[\alpha+\zeta\frac{\Tr(\sprod{w}{t})}{\sprodq{t}{t}}\in\o^\times\text{ and }\alpha+(\zeta-1)\frac{\Tr(\sprod{w}{t})}{\sprodq{t}{t}}\in\o^\times.\]
  Setting $w'=\zeta w+\alpha t$ we obtain
  \begin{align*}
    \sprodq{w'}{w'}&=\zeta^2\sprodq{w}{w}+\alpha\left({\alpha+\zeta\frac{\Tr(\sprod{w}{t})}{\sprodq{t}{t}}}\right)\sprodq{t}{t},\\ 
    \sprodq{w-w'}{w - w'}&=(\zeta-1)^2\sprodq{w}{w}+\alpha\left({\alpha+(\zeta-1)\frac{\Tr(\sprod{w}{t})}{\sprodq{t}{t}}}\right)\sprodq{t}{t},
  \end{align*}
  which implies $\sprodq{w'}{w'}\o = \sprodq{w - w'}{w - w'}\o = \sprodq t t \o = \Tr(\FP^i)$.

  Thus in all cases we obtain $\sprodq{w'}{w'}\o=\sprodq{w-w'}{w - w'}\o=\Tr(\mathfrak P^i)$. By \Cref{resceichlerexists}, $E_w\circ S(L)$ and $E_{w-w'}\circ S(L)$ exist, and by \Cref{wnorm}~(d) we have $E_{w'}\circ S(L)=E_{w-w'}\circ S(L)=S(L)$. Thus $E_w\circ S(L)=(E_{w-w'}\circ S(L))\circ(E_{w'}\circ S(L))=S(L)$ by \Cref{compeich}.
\end{proof}

\begin{proof}[Proof of \Cref{symmetriesgenerators}]
  The sufficient condition is \Cref{apptheorem2}.
  For the necessary condition we now consider the case where the residue field is $\FF_2$.
  Now let
  \[ i=\begin{cases}0,&\text{if $e$ is even},\\1,&\text{if $e$ is odd},\end{cases} \]%
  and consider the lattice $L=H(i)\perp H(i)$.
  Then one can show that $U(L)\ne S(L)$: In the case $K=\QQ_2$ this was shown in \cite[6.3]{Hayakawa1968}, but the argument also works for the field $K$ under consideration.
  (Note that in \cite{Hayakawa1968} the case ``$e$ even'' is referred to as the ``ramified unit'' case and ``$e$ odd'' is referred to as the ``ramified prime'' case.)
\end{proof}

\bibliographystyle{amsplain}
\bibliography{unitary}

\end{document}